\documentclass[a4paper, reqno, 10pt]{amsart}

\usepackage[utf8]{inputenc}
\usepackage{amsthm,amsfonts,amssymb,amsmath,amsxtra,amsrefs}
\usepackage{mathtools}
\usepackage{bm}
\usepackage{latexsym}
\usepackage{stmaryrd}
\usepackage{graphicx}
\usepackage{tikz}
\usepackage{tikz-cd}
\usepackage{todonotes,cancel}
\usepackage[margin=1.25in]{geometry}
\usepackage{mathrsfs}
\usepackage{bm}
\usepackage{comment}

\usepackage[all]{xy}
\SelectTips{cm}{}
\usepackage{xr-hyper}
\usepackage[colorlinks=false,
   citecolor=Black,
   linkcolor=Red,
   urlcolor=Blue]{hyperref}
\usepackage{verbatim}
\RequirePackage{xspace}
\RequirePackage{etoolbox}
\RequirePackage{varwidth}
\RequirePackage{enumitem}
\RequirePackage{tensor}
\RequirePackage{mathtools}
\RequirePackage{longtable}
\RequirePackage{multirow}

\usepackage{rotating}

\tikzset{
    labl/.style={anchor=south, rotate=90, inner sep=.5mm}
}

\newtheorem{thm}{Theorem}

\newtheorem{thmintro}{Theorem}

\newtheorem{prop}[thm]{Proposition}

\newtheorem{lem}[thm]{Lemma}

\newtheorem{conjecture}[thm]{Conjecture}
\newtheorem{cor}[thm]{Corollary}

\theoremstyle{definition}

\newtheorem{defi}[thm]{Definition}
\newtheorem{example}[thm]{Example}

\newtheorem{rem}[thm]{Remark}
\newtheorem{remark}[thm]{Remark}

\numberwithin{equation}{section}
\numberwithin{thm}{section}

\newcommand{\BA}{\ensuremath{\mathbb {A}}\xspace}

\newcommand{\BC}{\ensuremath{\mathbb {C}}\xspace}

\newcommand{{\BG}}{\ensuremath{\mathbb {G}}\xspace}

\newcommand{{\BK}}{\ensuremath{\mathbb {K}}\xspace}

\newcommand{\BQ}{\ensuremath{\mathbb {Q}}\xspace}
\newcommand{\BR}{\ensuremath{\mathbb {R}}\xspace}

\newcommand{\BZ}{\ensuremath{\mathbb {Z}}\xspace}

\newcommand{\CA}{\ensuremath{\mathcal {A}}\xspace}
\newcommand{\CB}{\ensuremath{\mathcal {B}}\xspace}

\newcommand{\CD}{\ensuremath{\mathcal {D}}\xspace}

\newcommand{\CO}{\ensuremath{\mathcal {O}}\xspace}

\newcommand{\CR}{\ensuremath{\mathcal {R}}\xspace}

\newcommand{\CU}{\ensuremath{\mathcal {U}}\xspace}

\newcommand{\CZ}{\ensuremath{\mathcal {Z}}\xspace}

\newcommand{\SA}{\ensuremath{\mathscr {A}}\xspace}

\newcommand{\x}{\times}
\newcommand{\ox}{\otimes}

\newcommand{\bfe}{\textnormal{\textbf{e}}}

\newcommand{\bfs}{\textnormal{\textbf{s}}}
\newcommand{\bfu}{\textnormal{\textbf{u}}}
\newcommand{\bfv}{\textnormal{\textbf{v}}}
\newcommand{\bfw}{\textnormal{\textbf{w}}}

\newcommand{\id}{\textnormal{id}}
\newcommand{\uf}{\textnormal{uf}}
\newcommand{\sgn}{\textnormal{sgn}}
\newcommand{\Conf}{\textnormal{Conf}}
\newcommand{\Br}{\textnormal{Br}}
\newcommand{\ve}{\varepsilon}
\newcommand{\isoto}{\xrightarrow{\sim}}

\let\emptyset\varnothing

\begin{document}

\title[]{Upper cluster structure on Kac--Moody Richardson varieties}

\author[Huanchen Bao]{Huanchen Bao}
\address{Department of Mathematics, National University of Singapore, Singapore.}
\email{huanchen@nus.edu.sg}

\author[Jeff York Ye]{Jeff York Ye}
\address{Department of Mathematics, National University of Singapore, Singapore.}
\email{e1124873@u.nus.edu}
\subjclass[2020]{} 

\begin{abstract}
 We show coordinate rings of open Richardson varieties are upper cluster algebras for any symmetrizable Kac--Moody type. We further show the coordinate rings of (generalized) open Richardson varieties on the twisted product of flag varieties are upper cluster algebras for any symmetrizable Kac--Moody type. This includes, as special cases, reduced double Bruhat cells, Bott-Samelson varieties, braid varieties. Our results generalize various results by Casals--Gorsky--Gorsky--Le--Shen--Simental  and Galashin--Lam--Sherman-Bennett--Speyer in finite types.
\end{abstract}

\maketitle
	
\tableofcontents
 
\section{Introduction}
\label{sec:intro}

\subsection{Richardson varieties}

Let $G$ be a symmetrizable (minimal) Kac--Moody group over $\mathbb{C}$. Let $B^+$ and $B^-$ be opposite Borel subgroups. We denote the (thin) flag variety by $\mathcal{B} = G/B^+$. The flag variety admits a natural decomposition into open Richardson varieties 
\[
 \mathcal{B} = \sqcup_{v \le w} \mathring{\CB}_{v,w}, \quad \text{ where $v \le w $ are in the Weyl group }W.
\]
The main result of this paper is the following theorem. 

\begin{thmintro}[Theorem~\ref{RichardsonCluster}]\label{thm:1}
The coordinate ring $\mathbb{C}[\mathring{\mathcal{B}}_{v,w}]$ of the open Richardson variety in $\CB$ is an upper cluster algebra.
\end{thmintro}

Cluster structures on open Richardson varieties for reductive groups have been conjectured by Leclerc \cite{Lec}. Such a cluster structure has been established in the recent breakthroughs by Casals--Gorsky--Gorsky--Le--Shen--Simental \cite{CGG+} and Galashin--Lam--Sherman-Bennett--Speyer \cites{GLSS, GLSS2}. We shall refer to {\em loc. cit.} for background and history on cluster structures of Richardson varieties.

Both \cite{CGG+} and \cite{GLSS2} establish the cluster structure on open braid varieties first, and deduce the cluster structure on open Richardson varieties as special cases. It is not known whether open Richardson varieties can be identified with open braid varieties in the Kac--Moody setting in general, so approaches in \cites{CGG+, GLSS2} do not carry to the Kac--Moody setting. Let us also mention that there has been attempts by Kashiwara--Kim--Oh--Park \cite{KKOP} generalizing Leclerc's categorical construction to the Kac--Moody setting.

If $G$ is reductive, our initial cluster seed $\bfs(v,\bfw)$ (Definition~\ref{RichardsonSeedDefi}) is the same as the one constructed in \cites{CGG+, Men22, GLSS2} (see \cite{CGG+25} for the comparison). The construction of our cluster seed is motivated by the construction of M\'enard \cite{Men22} in the reductive simply-laced case. We simplify and generalize M\'enard's construction to the Kac--Moody setting. Moreover, we give a geometric explanation for each step in the seed construction.

\subsection{Twisted product of flag varieties}
Let 
\[
\mathcal{Z}  = G \times ^{B^+} G \times ^{B^+} G \times^{B^+} \cdots \times ^{B^+} G/B^+ \text{ ($n$ factors) }
\]
be the twisted product of flag varieties. 
For any sequence $\underline{w}= (w_1, \cdots, w_n)$ of elements in $W$, we define 
\[
\mathring{\CZ}_{\underline{w}}=B^+ \dot w_1 B^+ \times^{B^+} B^+ \dot w_2 B^+ \times^{B^+} \cdots \times^{B^+} B^+ \dot w_n B^+/B^+.
\]
This is analogous to Schubert cell in the flag variety $\CB$. For any $v \in W$, we define  $\mathring{\CZ}^v$ as $m ^{-1}(B^- \dot v B^+/B^+)$, where $m: \CZ \to \CB$ is the convolution product. This is analogous to the opposite Schubert cell in the flag variety $\CB$. The open Richardson variety of $\CZ$ is defined by \[
\mathring{\CZ}_{v, {\underline{w}}}=\mathring{\CZ}_{{\underline{w}}} \cap 
\mathring{\CZ}^v.
\]
The intersection is non-empty if and only if $v \le m_{\ast}({\underline{w}})$ \cite{BH24b}. Here $m_{\ast}({\underline{w}})$ denotes the Demazure product of $\underline{w}$. These open Richardson varieties include, as special cases, the reduced double Bruhat cells, the Bott-Samelson varieties, and the braid varieties. 

\begin{thmintro}[Theorem~\ref{TwistedProductCluster} \& Remark~\ref{rem:twisted}]
The coordinate ring $\mathbb{C}[\mathring{\mathcal{Z}}_{v,{\underline{w}}}]$ of the open Richardson variety $\mathring{\mathcal{Z}}_{v,{\underline{w}}}$ in the twisted product is an upper cluster algebra.
\end{thmintro}

Our construction of the cluster structure on $\mathring{\mathcal{Z}}_{v,{\underline{w}}}$ is based on the thickening considered by the first author and He in \cite{BH24b}. Via this thickening, the upper cluster structure on $\mathbb{C}[\mathring{\mathcal{Z}}_{v,{\underline{w}}}]$ is actually a direct consequence of the upper cluster structure on a Richardson variety of another more complicated Kac--Moody group from Theorem~\ref{thm:1}. In the case of braid varieties, we expect our seed to coincide with the seed constructed in \cite{CGG+} and \cite{GLSS2}. Logically, our construction is reverse to the constructions in \cites{CGG+, GLSS2} when restricting to reductive groups.

\subsection{Total positivity}

The totally nonnegative part  ${\mathcal{B}}_{v,w, >0}$ of the open Richardson variety for reductive groups is defined by Lusztig \cite{Lus94} using his theory of canonical bases. The generalization to the Kac--Moody setting is straightforward \cite{BH24a}. On the other hand, the totally nonnegative part ${\mathcal{Z}}_{v,{\underline{w}}, >0}$ of the Richardson variety on the twisted product is defined by the first author and He \cite{BH24b}.

One of the original motivation of Fomin--Zelevinsky \cite{FZ02}  to introduce cluster algebras is to better understand Lusztig's theory of canonical bases and total positivity. Given the cluster structure on $\mathring{\mathcal{B}}_{v,w}$ (resp. $\mathring{\mathcal{Z}}_{v,{\underline{w}}}$), we can define the cluster positive locus ${\mathcal{B}}^{cl}_{v,w, >0}$ (resp. ${\mathcal{Z}}^{cl}_{v,{\underline{w}},>0}$) of $\mathring{\mathcal{B}}_{v,w}$ (resp. $\mathring{\mathcal{Z}}_{v,{\underline{w}}}$), where all cluster variables take values in $\mathbb{R}_{>0}$.

\begin{thmintro}[Proposition~\ref{TPAgree} \& Proposition~\ref{TPAgree2}]\label{thm:3}The two notions of positivity coincide, that is, 
\[
{\mathcal{B}}^{cl}_{v,w,>0} ={\mathcal{B}}_{v,w, >0}
,\quad
{\mathcal{Z}}^{cl}_{v,{\underline{w}},>0}={\mathcal{Z}}_{v,{\underline{w}}, >0}.
\]
\end{thmintro}

Theorem~\ref{thm:3} is new even in the reductive case in general. Such identification has previously only been established \cites{GL23,SSB} for open Richardson varieties in finite type A.  Thanks to the concrete geometric interpretation of the seed construction, Theorem~\ref{thm:3} is again a direct consequence of Theorem~\ref{thm:1}.

\subsection{Strategy of the proof} 
Two key ingredients of our approach to Theorem~\ref{thm:1} are the Bruhat atlas introduced by He--Knutson--Lu \cite{HKL} and rational quasi-cluster maps introduced in \S\ref{sec:subalg}. 

The Bruhat atlas used in our setting actually goes back to Kazhdan--Lusztig \cite{KL}. For $v \le r \le w$ in the Weyl group $W$, we have an isomorphism 
\[
\mathring{\CB}_{v,w}\cap \dot{r}U^-B^+/B^+ \isoto\mathring{\CB}_{v,r}\x \mathring{\CB}_{r,w}.
\]
Note that $\mathring{\CB}_{v,r} \cong \mathbb{C}^\x$, if it is one-dimensional. Iterating the process, we can realize $(\mathbb{C}^\x )^n \times \mathring{\CB}_{v,w}$ as an open subvariety of $\mathring{\CB}_{e,w}$. We can further identify $\mathring{\CB}_{v,w}$ as a locally closed subvariety in $\mathring{\CB}_{e,w}$ by specializing the coordinates in  $(\mathbb{C}^\x )^n$.

The cluster structure on $\mathring{\CB}_{e,w}$ has already been established following Williams \cite{Wil13} and Shen-Weng \cite{SW21}. The geometric construction of localization and specialization roughly correspond to freezing and deletion in terms of cluster algebras, respectively.  
In order to perform the relevant freezing and deletion, it remains to calculate explicitly \begin{tikzcd}[column sep=small]T_i: \mathring{\CB}_{e,w} \ar[r,dashed]& (\mathbb{C}^\x)^n \ar[r] & \mathbb{C}^\x\end{tikzcd} to each coordinate in terms of cluster variables.

We introduce rational quasi-cluster maps \S\ref{sec:subalg} for this calculation.  Roughly speaking, rational quasi-cluster maps are rational maps between two (upper) cluster algebras generalizing quasi-equivalence. The key feature is that rational quasi-cluster maps preserve Laurent monomials of cluster variables. 
  \vspace{.2cm}
  
\noindent {\bf Acknowledgment: } HB is supported by MOE grants A-0004586-00-00 and A-0004586-01-00. We would like to thank Ian Le and Melissa Sherman-Bennett for explaining relations between the constructions in \cite{CGG+} and \cite{GLSS2}.

\section{Preliminaries}\label{sec:prelim}

\subsection{Minimal Kac--Moody groups}\label{sec:prelim:KM}

We follow \cite{Kum02} for basics on Kac--Moody groups. Let $I$ be a finite set and $A=(a_{ij})_{i,j\in I}$ be a symmetrizable generalized Cartan matrix. A \textit{Kac--Moody root datum} associated to $A$ is a quintuple
\[
\CD=(I,A,X,Y,(\alpha_i)_{i\in I},(\alpha_i^\vee)_{i\in I}),
\]
where $X$ is a free $\BZ$-module of finite rank with $\BZ$-dual $Y$, and the elements $\alpha_i$ of $X$ and $\alpha_i^\vee$ of $Y$ are such that $\langle \alpha_j^\vee,\alpha_i\rangle=a_{ij}$. 
We will always work with simply connected root datum, i.e., $Y=\BZ[\alpha_i^\vee]_{i\in I}$ and $X=\BZ[\omega_i]_{i\in I}$, where $\omega_i$ denotes the fundamental weight. We define the set of dominant weights $X^+=\{\lambda\in X\,|\, \langle \alpha_i^\vee,\lambda\rangle \geq0 \text{ for all $i\in I$}\}$. Note that a simply connected Kac--Moody root datum is uniquely determined by its generalized Cartan matrix. 

Let $W$ be the Weyl group associated to $\CD$ generated by the simple reflection $s_i$ for $i\in I$. We denote the length function by $\ell(\cdot)$. The group $W$ acts naturally on both $X$ and $Y$. Let $\Delta^{re}=\{w(\pm\alpha_i)\,|\, i\in I, w\in W\}\subset X$ be the set of real roots. Then $\Delta^{re}=\Delta^{re}_+\sqcup \Delta^{re}_-$ is a disjoint union of the positive and negative roots.

The \textit{minimal Kac--Moody group} $G$ associated to the Kac--Moody root datum $\CD$ over $\BC$ is the group generated by the torus $T=Y\ox_\BZ \BC^\x$ and the root subgroup $U_\alpha\simeq \BC$ for each real root $\alpha$, subject to the Tits relations \cite{Tit87}. Let $U^+\subset G$ (resp. $U^-\subset G$) be the subgroup generated by $U_\alpha$ for $\alpha\in \Delta^{re}_+$ (resp. $\alpha\in \Delta^{re}_-$). Let $B^\pm\subset G$ be the subgroup generated by $T$ and $U^\pm$, respectively.

For each $i\in I$, we fix isomorphisms $x_i:\BC\to U_{\alpha_i}$, $y_i:\BC\to U_{-\alpha_i}$ such that the maps
\[
\begin{pmatrix}
1 & a \\
0 & 1 \\
\end{pmatrix}\mapsto x_i(a), \quad 
\begin{pmatrix}
b & 0 \\
0 & b^{-1} \\
\end{pmatrix}\mapsto \alpha_i^\vee(b), \quad 
\begin{pmatrix}
1 & 0 \\
c & 1 \\
\end{pmatrix}\mapsto y_i(c),
\]
define  a homomorphism $SL_2\to G$. The data $(T,B^+,B^-,x_i,y_i;i\in I)$ is called a \textit{pinning} for $G$.

For each $i\in I$, define $\dot{s}_i=x_i(1)y_i(-1)x_i(1)\in G$. For any $w\in W$ with reduced expression $w=s_{i_1}\cdots s_{i_n}$, we define $\dot{w}=\dot{s}_{i_1}\cdots \dot{s}_{i_n}$. It is known that $\dot{w}$ is independent of the choice of reduced expressions.

\subsection{Combinatorics on the Weyl group}
\label{sec:prelim:Weyl}

An \textit{expression} is a finite sequence $\bfw=(i_1,\cdots,i_n)$ in $I$. Then we write $w=s_{i_1}\cdots s_{i_n}\in W$, and say that $\bfw$ is an expression for $w$. By abuse of notations, we also write the expression as $\bfw=s_{i_1}\cdots s_{i_n}$. We define the length $\ell(\bfw)=n$. The expression $\bfw$ is said to be \textit{reduced} if the length $n$ is minimal among all expressions for $\bfw$, in which case we we have $\ell(w)=\ell(\bfw)$. We further write $w_{(k)}=s_{i_1}\cdots s_{i_k}$.

A \textit{subexpression} of $\bfw$ is a sequence $\bfw'=(i'_1,\cdots,i'_n)$ such that $i'_k\in \{i_k,e\}$ for all $k$. 
We further write $w'=s'_{i_1}\cdots s'_{i_n}\in W$, where $s'_{i_k}=s_{i_k}$ if $i'_k=i_k$, and $s'_{i_k}=e$ if $i'_k=e$. We say that $\bfw'$ is a subexpression for $w'$ in $\bfw$.

For any $w'\in W$, recall that $w'\leq w$ with respect to the Bruhat order of $W$ if and only if there exist a subexpression for $w'$ in any reduced expression $\bfw$ for $w$. The right weak Bruhat order $\leq_R$ on $W$ is the partial order generated by the covering relations $w\lessdot_R ws_i$ for $w<ws_i$ and $i\in I$. We similarly define the left weak Bruhat order $\lessdot_L$.

Let $\Br^+ =\Br^+_W$ be the associated (positive) braid monoid of $W$. It is the monoid generated by $s_i$ for $i\in I$ subject to the braid relations in $W$. Expressions for $\beta \in \Br^+$ are defined similar to expressions of elements in $W$. We can view $W$ as a subset of $\Br^+$ using the reduced expressions.

The Demazure product on the set $W$ is defined by
\begin{itemize}
    \item $x*y=xy$ if $x,y\in W$ such that $\ell(xy)=\ell(x)+\ell(y)$;
    \item $s_i*w=w$ if $i\in I$, $w\in W$ such that $s_iw<w$.
\end{itemize}
Then $(W,*)$ is a monoid. We also define $m_*(\bfw)=s_{i_1}*\cdots *s_{i_n}$ for any expression $\bfw$. It is easy to check that $m_*$ descends to a monoid homomorphism $m_*:\Br^+\to W$ such that $m_*(s_i)=s_i$ for all $i\in I$.

Let $\bfw = (i_1, \dots, i_n)$ be an expression and $v\in W$ be such that $v \leq m_*(\bfw)$. Let $\bfv = (i'_1, \dots, i'_n)$ be a subexpression of $v$ in $\bfw$. The subexpression is called {\em distinguished} if $v_{(k)} \le v_{(k-1)}s_{i_k}$ for all $k$. The subexpression $\bfv$ is called \textit{left-most} \cite{BH24b}*{\S3.2} if it is defined as follows:
\begin{itemize}
\item If $s_{i_1}v<v$, then $i'_1=i_1$ and we let $(i'_2,\cdots,i'_n)$ to be the left-most subexpression for $s_{i_1}v$ in $(i_2,\cdots,i_n)$.
\item If $s_{i_1}v>v$, then $i'_1=e$ and we let $(i'_2,\cdots,i'_n)$ to be the left-most subexpression for $v$ in $(i_2,\cdots,i_n)$.
\end{itemize}
We can similarly define the \textit{right-most} subexpression by starting from the right side of $\bfw$. Note that after removing the $e$'s, both the left-most and right-most subexpressions are reduced expressions for $v$. The right-most subexpression is always distinguished.

Let $\overline{I}=\{\overline{i}\,|\, i\in I\}$. A \textit{signed expression} is a finite sequence $(i_1,\cdots,i_n)$ with $i_k\in I\sqcup\overline{I}$.

\subsection{Flag varieties}\label{sec:prelim:flag}

We denote by $\CB=G/B^+$ the (thin) full flag variety of $G$. For $v,w\in W$, let $\mathring{\CB}_w=B^+\dot{w}B^+/B^+$ be the \textit{Schubert cell}, $\mathring{\CB}^v=B^-\dot{v}B^+/B^+$ be the \textit{opposite Schubert cell}, and $\mathring{\CB}_{v,w}=\mathring{\CB}_w\cap \mathring{\CB}^v$ be the \textit{open Richardson variety}. We have the decompositions
\[
\CB=\sqcup_{w\in W} \mathring{\CB}_w=\sqcup_{v\in W} \mathring{\CB}^v=\sqcup_{v,w\in W} \mathring{\CB}_{v,w}.
\]
By \cite{Kum17}, the intersection $\mathring{\CB}_{v,w}$ is non-empty if and only if $v\leq w$, in which case $\mathring{\CB}_{v,w}$ is irreducible and smooth with dimension $\ell(w)-\ell(v)$.

We will also work with the \textit{(opposite) flag variety} $G/B^-$, the \textit{decorated flag variety} $G/U^+$, the \textit{(opposite) decorated flag variety} $G/U^-$. There are   natural projection maps $\pi:G/U^\pm\to G/B^\pm$. We will use lower subscripts $B_?$ or $U_?$ to denote flags in $G/B^+$ or $G/U^+$, and use upper subscripts $B^?$ or $U^?$ to denote flags in $G/B^-$ or $G/U^-$, respectively.

For a pair of flags $(B_0=xB^+, B_1=yB^+)$, we define the \textit{distance} $d(B_0,B_1)=w\in W$ if $x^{-1}y\in B^+\dot{w}B^+$, also written as $B_0 \xrightarrow{w}B_1$. Similarly, for the pair $(B^0=xB^-,B^1=yB^-)$, we define the \textit{distance} $d(B^0,B^1)=w\in W$ if $x^{-1}y\in B^-\dot{w}^{-1}B^-$, also written as $B^0 \xrightarrow{w}B^1$. For the pair $(B^0=xB^-, B_0=yB^+)$, we define the \textit{codistance} $d(B^0,B_0)=w\in W$ if $x^{-1}y\in B^-\dot{w}B^+$. We say that $(B^0,B_0)$ are opposite if $d(B^0,B_0)=1$. These definitions extends to decorated flags via the projection map. For $U_0=xU^+$ and $U_1=yU^+$, we say that $(U_0,U_1)$ are \textit{compatible} if $x^{-1}y\in U^+\dot{w}U^+$ for some $w\in W$. For $U^0=xU^-$ and $U^1=yU^-$, we say that $(U^0,U^1)$ are \textit{compatible} if $x^{-1}y\in U^-\dot{w}^{-1}U^-$ for some $w\in W$. These definitions follow \cite{SW21}*{\S2}.

\subsection{Generalized minors}\label{sec:prelim:minor}

Let $\lambda\in X^+$ and $L(\lambda)$ be the integrable highest weight $G$-module with highest weight $\lambda$. Fix a highest weight vector $\eta_\lambda$. The extremal weight spaces $L(\lambda)_{w\lambda}$ are all one-dimensional, spanned by $\eta_{w\lambda}=\dot{w}\cdot \eta_\lambda$.

For any $\xi\in L(\lambda)$ and $w\in W$, let $\langle \xi,\dot{w}\cdot \eta_\lambda\rangle$ be the coefficient of $\dot{w}\cdot \eta_\lambda$ in $\xi$. We define the minor $\Delta_\lambda : G \to \BC$ by $\Delta_\lambda(g)=\langle g \eta_\lambda,\eta_\lambda\rangle$. For any $w,w'\in W$, we define the generalized minor $\Delta^{w\lambda}_{w'\lambda}(g)=\langle g \dot{w}'\cdot \eta_\lambda,\dot{w}\cdot \eta_\lambda\rangle$ as a function on $G$.

For $x\in U^+$, $y\in U^-$, and $t\in T$, we have 
\begin{align*}
\Delta_\lambda(gx)=\Delta_\lambda(yg)=\Delta_\lambda(x), \quad 
\Delta_\lambda(gt)=\Delta_\lambda(tg)=\langle t,\lambda\rangle\Delta_\lambda(x).
\end{align*}
We also have $\Delta^{w\lambda}_{w'\lambda}(g)=\Delta_\lambda(\dot{w}^{-1}g\dot{w}')$. Note that 
\begin{equation}\label{eq:well}
\frac{\Delta^{w\lambda}_\lambda(gb)}{\Delta^{w'\lambda}_\lambda(gb)}=\frac{\Delta^{w\lambda}_\lambda(g)}{\Delta^{w'\lambda}_\lambda(g)}, \quad \text{ for any $b\in B^+$}.
\end{equation}
Therefore $\Delta^{w\lambda}_\lambda/\Delta^{w'\lambda}_\lambda$ is well-defined on the flag variety $\CB$.

For decorated flags $U^0=xU^-$ and $U_0=yU^+$, we define $\Delta_\lambda(U^0,U_0)=\Delta_\lambda(x^{-1}y)$. It is clear this is a well-defined function from $G/U^- \times G/U^+$ to $\BC$.

\subsection{Cluster algebras}\label{sec:prelim:cluster}

A \textit{(labeled) seed} $\bfs$ is a quadruple $(J, J_\uf,(\ve_{ij})_{i,j\in J}, (d_i)_{i\in J})$ consists of the following data:
\begin{itemize}
    \item An index set $J$ with $|J|=n$, together with a subset $J_\uf\subset J$. The elements in $J$ are called vertices. The vertices in $J_\uf$ are called \textit{mutable} or \textit{unfrozen}. The vertices in $J-J_\uf$ are called \textit{frozen}.
    \item A skew-symmetrizable matrix $(\ve_{ij})_{i,j\in J}$ over $\BQ$ with $\ve_{ij}\in \BZ$ if at least one of $i,j$ is in $J_\uf$, called the \textit{exchange matrix}.
    \item Integers $d_i$ for $i\in J$ with $\gcd(d_i)=1$ such that $\ve_{ij}d_j=-\ve_{ji}d_i$, called the \textit{multipliers}.
\end{itemize}

Let $[a]_+=\max(a,0)$ for $a \in \BQ$. For any $k\in J_\uf$, the \textit{mutation at $k$} gives a new seed $\bfs'=\mu_k(\bfs)$ with the same $J$, $J_\uf$ and $(d_i)_{i \in J}$, and a new exchange matrix $(\ve'_{ij})_{i,j\in J}$ given by
\[
\ve'_{ij} = \begin{cases}-\ve_{ij}, & \text{if } k \in \{i,j\};\\
\ve_{ij}+[\ve_{ik}]_+\ve_{kj}+\ve_{ik}[-\ve_{kj}]_+, &\text{otherwise}.
\end{cases}
\]

We can repeat this process mutating at arbitrary $k\in J_\uf$. A seed $\bfs'$ is said to be \textit{mutation equivalent} to $\bfs$, denoted by $\bfs'\sim \bfs$, if it can be obtained from $\bfs$ by a sequence of mutations.

For any seed $\bfs_0$, the (abstract) cluster algebra and upper cluster algebra with initial seed $\bfs_0$ are defined as follows. Fix an ambient field $K = \BC(A_1, \dots, A_n)$ the field of rational functions over $\BC$ in $n$ variables, generated by the formal variables $A_i$ for $i\in J$. The variables $\{A_{i,\bfs_0}=A_i\}$ are called the \textit{cluster variables} in the seed $\bfs_0$. The variables $\{A_i \,|\, i\in J-J_\uf \}$ are  called \textit{frozen variables}.

Given the cluster variables $\{A_{i,\bfs}\}_{i \in J}$ in a seed $\bfs\sim \bfs_0$, the cluster variables $\{A_{i,\bfs'}\}_{i \in J}$ in $\bfs'=\mu_k(\bfs)$ are elements of $K$ defined by 
\[
\begin{cases}
A_{i,\bfs'}=A_{i,\bfs}, &\text{for $i\neq k$};\\
A_{k,\bfs'} = \frac{1}{A_{k,\bfs}}\left(\prod_{i\in J}A_{i,\bfs}^{[\ve_{ki}]_+}+\prod_{i\in J}A_{i,\bfs}^{[-\ve_{ki}]_+}\right), &\text{for $i = k$.}
\end{cases}
\]
Note that for $i\in J-J_\uf$, $A_{i,\bfs}=A_{i,\bfs_0}=A_i$ for any $\bfs\sim \bfs_0$. Each $\{A_{i,\bfs}\}_{i\in J}$ form a set of free generators of $K$.

The \textit{cluster algebra} $\CA(\bfs_0)$ is defined as the $\BC$-subalgebra of $K$ generated by the cluster variables $A_{i,\bfs}$ over all $i\in J$ and $\bfs\sim \bfs_0$, and $A_i^{-1}$ for $i\in J-J_\uf$. The \textit{upper cluster algebra} $\CU(\bfs_0)$ is defined as the intersection
\[
\CU(\bfs_0)=\bigcap_{\bfs\sim \bfs_0}\BC[A_{i,\bfs}^{\pm1}]_{i\in J}.
\]

It is clear that mutation equivalent seeds give canonically isomorphic (upper) cluster algebras. So the construction only depends on the mutation equivalence class of $\bfs_0$. The Laurent phenomenon introduced and proved in \cite{FZ02} asserts that $\CA(\bfs_0)\subset \CU(\bfs_0)$. The conditions for $\CA(\bfs_0)=\CU(\bfs_0)$ is a long-standing problem in the study of cluster algebras.

\begin{rem}
We will often omit the seed $\bfs$ in notations if no confusion arises. We often deal with multiple seeds $\bfs,\bfs', \bfs''$,  which are not necessarily mutation equivalent. In that case, we use the prime symbol, for example $J', J''$ and $A'_i, A_i''$, to highlight which seed the object is related to.
\end{rem}

\begin{rem}\label{rem:iden}
For an explicit $\BC$-algebra $R$, we say that $R$ is a cluster algebra or an upper cluster algebra if $R\simeq \CA(\bfs_0)$ or $R\simeq \CU(\bfs_0)$ for some seed $\bfs_0$. In this case, we shall always identify the ambient field $K$ with the fraction field ${\rm Frac}(R)$ of $R$ and consider $A_{i,\bfs}$ as elements of $R$. 
\end{rem}

An affine variety $V$ over $\BC$ is said to have an \textit{(upper) cluster structure} given by a seed $\bfs$ if its coordinate ring $\BC[V]$ is an (upper) cluster algebra isomorphic to $\CU(\bfs)$. Then cluster variables will be identified as regular functions on $V$.

\subsection{Operations on seeds}
\label{sec:prelim:operations}

We follow \cite{FG06} for the amalgamation and the defrosting, and follow \cite{Mul13} for the freezing and the deletion.

Let $\bfs$, $\bfs'$ be two seeds. Suppose there are $J_1\subset J-J_\uf$ and $J'_1\subset J'-J'_\uf$ with a bijection $\phi:J_1\to J'_1$ such that $d_i=d'_{\phi(i)}$ for all $i\in J_1$. Define the \textit{amalgamation} of $\bfs$ and $\bfs'$ to be a seed $\bfs''$ where
\begin{itemize}
    \item the index set $J''=J\sqcup J'/\sim_\phi$ with $J''_\uf=J_\uf\sqcup J'_\uf$;
    \item we first extend $(\ve_{ij})_{i,j\in J}$ and $(\ve'_{ij})_{i,j\in J'}$ to matrices indexed by $J''$ by adding zeroes, and then set $\ve''_{ij}=\ve_{ij}+\ve'_{ij}$;
    \item the multipliers $d''_i$ are inherited from either $d_i$ or $d'_i$. The condition $d_i=d'_{\phi(i)}$ ensures that $d''_i$ is well defined.
\end{itemize}

Let $\bfs$ be a seed. Suppose there is $i\in J-J_\uf$ such that $\ve_{ij}, \ve_{ji}\in \BZ$ for all $j\in J$. The \textit{defrosting} of $\bfs$ at $i$ is the seed obtained from $\bfs$ by moving $i$ to $J_\uf$. 
Conversely, for $i\in J_\uf$, the \textit{freezing} of $\bfs$ at $i$ is the seed $\bfs'$ obtained from $\bfs$ by moving $i$ from $J_\uf$ to $J-J_\uf$. We can identify the ambient fields of $\bfs$ and $\bfs'$.

\begin{prop}[\cite{Mul13}*{Proposition 3.1}]
\label{FreezingInclusion}
Let $\bfs'$ be the freezing of $\bfs$ at $i$. Then as subalgebras of the ambient field $K$, we have 
\[
\CA(\bfs')\subset \CA(\bfs)[A_{i,\bfs}^{-1}]\subset \CU(\bfs)[A_{i,\bfs}^{-1}]\subset \CU(\bfs').
\]
\end{prop}

For any $i\in J$, the \textit{deletion} of $\bfs$ at $i$ is the seed $\bfs^\dagger$ obtained from $\bfs$ by removing $i$. The ambient field of $\bfs^\dagger$ can be identified with a subfield of the ambient field $K$ of $\bfs$. In particular, $\CA(\bfs^\dagger)$ and $\CU(\bfs^\dagger)$ can be viewed as subalgebras of $K$.
Deletion is particularly nice on  frozen vertices. 

\begin{prop}[\cite{Mul13}*{Proposition 3.7}]
\label{IsolatedDeletion}
Let $\bfs^\dagger$ be the deletion of $\bfs$ at an frozen vertex $i$. Then the natural maps are isomorphisms
\[
\CA(\bfs)/(A_{i} = 1) \cong \CA(\bfs^\dagger), \quad \CU(\bfs)/(A_{i} = 1) \cong \CU(\bfs^\dagger).
\]
\end{prop}

\subsection{Lie theoretic quivers}\label{sec:prelim:Lie}

We now return to the Kac--Moody setting and follow \cites{GS24, SW21}. For any $i\in I$, we define the seed $\bfs(i) = (J(i), J(i)_{\uf}, (\ve_{ij})_{i,j \in J(i)}, (d_i)_{i \in J(i)})$ as follows:
\begin{itemize}
    \item the index set $J(i)=(I-\{i\})\cup \{i_l,i_r\}$ with $J(i)_\uf=\emptyset$;
    \item the multipliers are obtained from the symmetrizers $d_i$ of the generalized Cartan matrix, identifying $i_l,i_r\in J(i)$ with $i\in I$;
    \item the exchange matrix is given by $\ve_{i_lj}=-a_{ij}/2$ and $\ve_{i_rj}=a_{ij}/2$ for $j\neq i$, $\ve_{i_ri_l}=1$, and $\ve_{jk}=0$ for $i\not\in\{j,k\}$.
\end{itemize}

The seed $\bfs(\overline{i})$ for $\overline{i}\in \overline{I}$ is the same as $\bfs(i)$ but with exchange matrix $(-\ve_{ij})_{i,j \in J(i)}$. The seed $\bfs(\bfw)$ for an arbitrary signed expression $\bfw$ is constructed inductively by amalgamating the seeds $\bfs(i)$ or $\bfs(\overline{i})$ explained below.

Fix a signed expression $\bfw$. For any $i\in I$, let $n_i$ be the total number of times $i$ or $\overline{i}$ appears in $\bfw$. Suppose we are given the seed $\bfs(\bfw)$ with $J(\bfw)=\{(i,l)\,|\, i\in I, 0\leq l\leq n_i\}$, $J(\bfw)_\uf=\{(i,l)\,|\, i\in I, 0< l< n_i\}$, $d_{(i,l)}=d_i$ and $\ve_{(i,0)(j,0)}, \ve_{(i,n_i)(j,n_j)}\in\{a_{ij}/2,-a_{ij}/2\}$. Note that the starting point $\bfs(i)$ or $\bfs(\overline{i})$ satisfies these assumptions, by identifying $i_l$ with $(i,0)$ and $i_r$ with $(i,1)$. 
Here one should think of $(i,l) \in J(\bfw)$ corresponds to the $l$-th appearance of $i \in I$ counting (starting from $0$) from the left.

The seed $\bfs(\bfw,i)$ is defined as the amalgamation $\bfs(\bfw)$ and $\bfs(i)$, defrosting at $(i,n_i)$. More precisely, we identify $(i,n_i)\in \bfs(\bfw)$ with $i_l\in J(i)$, and $(j,n_j)\in \bfs(\bfw)$ with $j\in J(i)$ for $j\neq i$. In $\bfs(\bfw,i)$, the vertex $i_r$ will be renamed as $(i,n_i+1)$. It is easy to check that $\bfs(\bfw,i)$ also satisfies the assumptions in the previous paragraph. The seed  $\bfs(\bfw,\overline{i})$ is defined similarly. Thus we obtain $\bfs(\bfw)$ for any signed expression $\bfw$. In entirely similar ways, we can construct the seeds $\bfs(i,\bfw)$ and $\bfs(\overline{i},\bfw)$ from the seed $\bfs(\bfw)$.

We define $J^+(\bfw)=\{(i,l)\,|\, i\in I, 0< l\leq  n_i\}$ and $J^0(\bfw)=\{(i,0)\,|\, i\in I\}$. Then there is a bijection $J^+(\bfw)$ with $\{1,\cdots,\ell(w)\}$, where $(i,l)$ corresponds to the index $k$ such that $i_k$ is the $l$-th appearance of $i$ or $\overline{i}$ in $\bfw$. This induces an order on $J^+(\bfw)$, denoted by $\le$. We also impose $(i,0)<(j,l)$ for any $i,j\in I$, $l>0$. We often say vertices of the form $(i,l)$ are at level $i$. The reader is encouraged to keep this visualization, which would make various combinatorics more tractile.

For any seed $\bfs$, we define the sequence of mutations
\begin{equation}\label{eq:muiright}
\mu_{\vec{i}} = \mu_{(i,n_i-1)} \cdots \circ \mu_{(i,2)} \circ\mu_{(i,1)}.
\end{equation}
Intuitively, we mutate at all mutable vertices $(i,l)$ for the fixed $i\in I$ starting with $l=1$. In other words, we mutate all vertices at level $i$ from the left.

\subsection{Double Bott-Samelson varieties}\label{sec:prelim:dBS}

We follow \cite{SW21} for this subsection. 
Let $\beta,\beta'\in \Br^+$. We fix expressions $\bfu=(i_1,\cdots,i_m)$ and $\bfv=(j_1,\cdots,j_n)$ for $\beta$ and $\beta'$, respectively. The decorated Bott-Samelson cell $\Conf^\beta_{\beta'}(\SA)$ is the space of $G$-orbits of decorated flags satisfying the relative positions according to the following picture
\begin{center}
\begin{tikzcd}
U^0\ar[r,"s_{i_1}"]\ar[dash,d] & U^1 \ar[r,"s_{i_2}"] & \cdots \ar[r,"s_{i_{m-1}}"] & U^{m-1} \ar[r,"s_{i_m}"] & U^m\ar[dash,d]\\
U_0\ar[r,swap,"s_{j_1}"] & U_1 \ar[r,swap,"s_{j_2}"] & \cdots \ar[r,swap,"s_{j_{n-1}}"] & U_{n-1} \ar[r,swap,"s_{j_n}"] & U_n
\end{tikzcd}
\end{center}
such that $U^0,\cdots,U^m\in G/U^-$, $U_0,\cdots,U_n\in G/U^+$, the horizontal decorated flags are compatible and the vertical decorated flags are opposite.
Note that a different choice of expressions for $\beta,\beta'$ give the same $\Conf^\beta_{\beta'}(\SA)$ up to a canonical isomorphism.

\begin{thm}[\cite{SW21}*{Theorem 1.1}]
\label{dBSCluster}
$\Conf^{\beta}_{\beta'}(\SA)$ is a smooth affine variety with an (upper) cluster structure.
\end{thm}

A seed for the cluster structure can be summarized as follows. Let $\bfw = \bfw(\beta,\beta')$ be any signed expression obtained by shuffling together $\bfu$ and $\bfv$, using $\overline{i_k}$ instead $i_k$ for $\bfu$. Then we obtain a seed $\bfs(\bfw)$ following Section \ref{sec:prelim:Lie}. We next explain the isomorphism $\CU(\bfs(\bfw))\simeq \CO(\Conf^{\beta}_{\beta'}(\SA))$ for the cluster structure.

We add extra lines in the picture of decorated flags which breaks it into triangles of the form
\begin{center}
\begin{tikzcd}
& U^k\ar[dash,dr] & \\
U_{l-1} \ar[rr,"s_{j_l}"]\ar[dash,ru] & & U_l
\end{tikzcd} and 
\begin{tikzcd}
U^{k-1} \ar[rr,"s_{i_k}"]\ar[dash,rd] & & U^k\\
& U_l\ar[dash,ur] & \\
\end{tikzcd}
\end{center}
appearing in the order in $\bfw$. Then we have  cluster variables $A_{(i,0)}=\Delta_{\omega_i}(U^0,U_0)$ (recall Remark~\ref{rem:iden}). For $(i,l)$ with $l>0$, the vertex $(i,l)$ corresponds to a triangle above, and we set $A_{(i,l)}=\Delta_{\omega_i}(U^k,U_l)$. 

\begin{rem}
By the uniqueness of compatible lifts \cite{SW21}*{Lemma~2.10}, a point in $\Conf^\beta_{\beta'}(\SA)$ can be specified using the data
\begin{center}
\begin{tikzcd}
U^0\ar[r,"s_{i_1}"]\ar[dash,d] & B^1 \ar[r] & \cdots \ar[r] & B^{m-1} \ar[r,"s_{i_m}"] & B^m\ar[dash,d]\\
U_0\ar[r,swap,"s_{j_1}"] & B_1 \ar[r] & \cdots \ar[r] & B_{n-1} \ar[r,swap,"s_{j_n}"] & B_n
\end{tikzcd}
\end{center}
and each $G$-orbit has a unique representative of the form
\begin{center}
\begin{tikzcd}
U^-\ar[r,"s_{i_1}"]\ar[dash,d] & B^1 \ar[r] & \cdots \ar[r] & B^{m-1} \ar[r,"s_{i_m}"] & B^m\ar[dash,d]\\
tU^+\ar[r,swap,"s_{j_1}"] & B_1 \ar[r] & \cdots \ar[r] & B_{n-1} \ar[r,swap,"s_{j_n}"] & B_n
\end{tikzcd}
\end{center}
for some $t\in T$.
\end{rem}

When $\bfu,\bfv$ are reduced expressions for $w,w'\in W$, the picture can be further reduced to
\begin{center}
\begin{tikzcd}
U^-\ar[r,"w"]\ar[dash,d] & B^1\ar[dash,d]\\
tU^+\ar[r,"w'"] & B_1
\end{tikzcd}
\end{center}
and we will also write $\Conf^w_{w'}(\SA)$ instead of $\Conf^\beta_{\beta'}(\SA)$. If we further have $w = e$, we shall often write $\Conf_{w'}(\SA)= \Conf^w_{w'}(\SA)$.

There is an isomorphism $T\x \mathring{\CB}_{e,w}\isoto \Conf_{w}(\SA)$ which sends $(t,gB^+)$ to
\begin{center}
\begin{tikzcd}
& U^-\ar[dash,dr] & \\
tU^+ \ar[rr]\ar[dash,ru] & & gB^+
\end{tikzcd}
\end{center}
This allows us to identify $\mathring{\CB}_{e,w}$ with the closed subvariety of $\Conf_{w}(\SA)$ defined by $t=1_G$. We obtain the cluster structure on $\mathring{\CB}_{e,w}$ by removing the frozen variables indexed by $(i,0)\in I(\bfw)$ for $i\in I$ by Proposition~\ref{IsolatedDeletion}. For $gB^+=z\dot{w}B^+$ with $z\in U^+$, the compatible decorated flags in the bottom row are given by $U_k=z\dot{w}_{(k)}U^+$, $w_{(k)}=s_{i_1}\cdots s_{i_k}$. This gives an explicit description of the cluster variables following \cite{SW21}*{Theorem~1.1}.

\begin{lem}
\label{SchubertClusterVariable}
For $z\dot{w}B^+\in \mathring{\CB}_{e,w}$ with $z\in U^+$, the cluster variables $A_k$ in the seed $\bfs(\bfw)$ are given by $\Delta^{\omega_{i_k}}_{w_{(k)}\omega_{i_k}}(z)$.
\end{lem}

\begin{remark}
We essentially only need the cluster structure on double Bruhat cells constructed in \cite{Wil13}, instead of the more general double Bott-Samelson varieties. However, various constructions are more transparent using the diagrammatic version in \cite{SW21}. 
\end{remark}

\section{Flag varieties and Richardson varieties}
 
\subsection{Bruhat atlas}\label{sec:constr:atlas}

For any $r\in W$, we have isomorphisms
\begin{align*}
(\dot{r}U^-\dot{r}^{-1}\cap U^+)\x (\dot{r}U^-\dot{r}^{-1}\cap U^-)&\isoto \dot{r}U^-\dot{r}^{-1}, & (g_1,g_2)&\mapsto g_1g_2;\\
(\dot{r}U^-\dot{r}^{-1}\cap U^-)\x (\dot{r}U^-\dot{r}^{-1}\cap U^+)&\isoto \dot{r}U^-\dot{r}^{-1}, & (h_1,h_2)&\mapsto h_1h_2.
\end{align*}
We define the morphisms
\begin{align*}
\sigma_{r,-}: \dot{r}U^-\dot{r}^{-1}&\to \dot{r}U^-\dot{r}^{-1}\cap U^-, & g_1g_2&\mapsto g_2;\\
\sigma_{r,+}: \dot{r}U^-\dot{r}^{-1}&\to \dot{r}U^-\dot{r}^{-1}\cap U^+, & h_1h_2&\mapsto h_2.
\end{align*}
By \cite{KWY13}*{Lemma 2.2}, we have an isomorphism
\[
\sigma_r=(\sigma_{r,+},\sigma_{r,-})\isoto (\dot{r}U^-\dot{r}^{-1}\cap U^+)\x (\dot{r}U^-\dot{r}^{-1}\cap U^-).
\]
Projecting to the flag variety, we obtain an isomorphism
\[
c_r=(c_{r,+},c_{r,-}): \dot{r}U^-B^+/B^+\isoto \mathring{\CB}_r\x \mathring{\CB}^r, \quad g\dot{r}B^+\mapsto (\sigma_{r,+}(g)\dot{r}B^+,\sigma_{r,-}(g)\dot{r}B^+), \quad g\in \dot{r}U^-\dot{r}^{-1}.
\]
 Restricting to $\mathring{\CB}^v$ and $\mathring{\CB}_{v,w}$, we obtain isomorphisms
\begin{align}
\label{AtlasMap1}
\mathring{\CB}^v\cap \dot{r}U^-B^+/B^+&\isoto\mathring{\CB}_{v,r}\x \mathring{\CB}^r,\\
\label{AtlasMap2}
\mathring{\CB}_{v,w}\cap \dot{r}U^-B^+/B^+&\isoto\mathring{\CB}_{v,r}\x \mathring{\CB}_{r,w}.
\end{align}

Suppose $v\leq r\leq w$ with $\ell(r)=\ell(v)+1$. Let $\mathbf{ r}=s_{i_1}\cdots s_{i_n}$ be a reduced expression for $r$. Then $v=s_{i_1}\cdots \hat{s}_{i_k}\cdots s_{i_n}$ for some unique $k$. Let $v_{(k)}=s_{i_1}\cdots s_{i_{k-1}}$, $r_{(k)}=s_{i_1}\cdots s_{i_k}=v_{(k)}s_{i_k}$.
By \cite{MR04}*{Proposition 8.1}, we have an isomorphism $\mathring{\CB}_{v,r}\isoto\BC^\x$ given by
\[
z\dot{r}B^+\mapsto \Delta^{v_{(k)}\omega_{i_k}}_{r_{(k)}\omega_{i_k}}(z)^{-1}, \quad z \in U^+.
\]
Composing with (\ref{AtlasMap1}), we obtain a function
\begin{equation}\label{eq:Drrs}
D_{v,\mathbf{r}} = (\Delta^{v_{(k)}\omega_{i_k}}_{r_{(k)}\omega_{i_k}})^{-1} \circ c_{r,+}:\mathring{\CB}^v\cap \dot{r}U^-B^+/B^+\to \BC^\x.
\end{equation}

\begin{prop}
\label{AtlasRegular1}
$D_{v,\mathbf{r}}$ extends to a regular function $\mathring{\CB}^v\to \BC$ given by $x\dot{v}B^+\mapsto \Delta^{r_{(k)}\omega_{i_k}}_{v_{(k)}\omega_{i_k}}(x)$ for $x\in U^- \cap \dot{v} U^- \dot{v}^{-1}$. 
\end{prop}

\begin{proof}

Note that $\mathring{\CB}^v \cong U^- \cap \dot{v} U^- \dot{v}^{-1}$, so the function is well-defined. The varieties $\mathring{\CB}^v$ and $\dot{r}U^-B^+/B^+$ are stable under left multiplication by $U^-\cap \dot{r}U^-\dot{r}^{-1}$. For any $g\in U^-\cap \dot{r}U^-\dot{r}^{-1}$, we have $\dot{r}_{(k)}^{-1}g \dot{r}_{(k)}\in U^-$. So for any $x \in G$, we have 
\begin{align*}
\Delta_{v_{(k)}\omega_{i_k}}^{r_{(k)}\omega_{i_k}}(gx)&=\Delta_{\omega_{i_k}}(\dot{r}_{(k)}^{-1}g \dot{r}_{(k)}\dot{r}_{(k)}^{-1}x \dot{v}_{(k)})\\
&=\Delta_{\omega_{i_k}}(\dot{r}_{(k)}^{-1}x \dot{v}_{(k)})\\
&=\Delta_{v_{(k)}\omega_{i_k}}^{r_{(k)}\omega_{i_k}}(x).
\end{align*}
On the other hand, by the definition of $\sigma_{r,+}$, the function $D_{v,r}$ is also invariant under left multiplication by $U^-\cap \dot{r}U^-\dot{r}^{-1}$.

Therefore it suffices to assume that 
$x \dot{v} B^+ \in \mathring{\CB}_{v,r}$. Then by \cite{MR04}*{Proposition~5.2}, we have $x = \dot{v}_{(k)} y_{i_k}(t) \dot{v}_{(k)}^{-1}$ for some $t \in \BC^\x$. Then by direct computation, we have 
\[
 x \dot{v} B^+ = \dot{v}_{(k)} x_{i_k}(t^{-1}) \dot{v}_{(k)}^{-1}\dot{r} B^+.
\]
Then we compute 
\[
\Delta_{v_{(k)}\omega_{i_k}}^{r_{(k)}\omega_{i_k}}(\dot{v}_{(k)} y_{i_k}(t) \dot{v}_{(k)}^{-1})
= t = 
\Delta_{r_{(k)}\omega_{i_k}}^{v_{(k)}\omega_{i_k}}( \dot{v}_{(k)} x_{i_k}(t^{-1}) \dot{v}_{(k)}^{-1})^{-1}.
\]

The proposition is proved.
\end{proof}

\begin{cor}\label{cor:T1}
We have  $\BC[\mathring{\CB}_{v,w}]/(D_{v,\mathbf{r}} =1) \cong  \BC[\mathring{\CB}_{r,w}]$.
\end{cor}
\begin{proof}
We have 
$\BC[\mathring{\CB}_{r,w}] \cong \BC[\mathring{\CB}_{v,w}][D_{v,\mathbf{r}}^{-1}]/(D_{v,\mathbf{r}} =1) \cong  \BC[\mathring{\CB}_{v,w}]/(D_{v,\mathbf{r}} =1)$.
\end{proof}

\subsection{The function $T_{r,rs_i}$}\label{sec:T}

Let $r \le w$ be such that $r < r s_i \le w$ for some $i \in I$. 
Let $\bfw=(i_1,\cdots,i_n)$ be a reduced expression for $w$ and let $w_{(k)} = s_{i_1}\cdots s_{i_k}$ and $w_{(0)}=e$.

\begin{lem}\label{lem:rrs}
Let $r \le rs_i \le w$ for some $i \in I$. We have the following commutative diagram of rational maps 
\[
\begin{tikzcd}
\mathring{\CB}_{e,w} \ar[r,dashed, "c_{rs_i,+}"]  \ar[d, dashed, "c_{r,-}"] & \mathring{\CB}_{e,rs_i} \ar[d, dashed, "c_{r,-}"] \\
\mathring{\CB}_{r,w} \ar[r,dashed,"c_{rs_i,+}"] & \mathring{\CB}_{r,rs_i}.
\end{tikzcd}
\]
\end{lem}

\begin{proof}
Let $x \in \mathring{\CB}_{e,w}$ be in the open locus where all maps are well-defined. Then we have 
\[
c_{rs_i,+}(x) = g_1x, \text{ for some unique }g_1 \in U^- \cap (\dot{r}\dot{s}_i U^- \dot{s}_i^{-1} \dot{r}^{-1}),
\]
\[
c_{r,-}(g_1x) = g_2g_1x \text{ for some unique }g_2 \in U^+ \cap \dot{r} U^- \dot{r}^{-1}.
\]
Recall the isomorphism $U^+ \times U^- \xrightarrow{\sim} U^- \times U^+$. We have $g_2g_1 = g_1' g_2'$ for $g_1' \in U^-$ and $g_2' \in U^+$.

Since $r \le_R r s_i$, we have $U^- \cap (\dot{r}\dot{s}_i U^- \dot{s}_i^{-1} \dot{r}^{-1}) \subset U^- \cap (\dot{r}  U^-  \dot{r}^{-1})$. The isomorphism above  restricts to 
\[
(U^+ \cap \dot{r} U^- \dot{r}^{-1}) \times (U^- \cap \dot{r} U^- \dot{r}^{-1}) \xrightarrow{\sim} (U^- \cap \dot{r} U^- \dot{r}^{-1}) \times (U^+ \cap \dot{r} U^- \dot{r}^{-1}).
\]
We conclude $g_2' \in  U^+ \cap \dot{r} U^- \dot{r}^{-1}$. Similarly, since $U^+ \cap \dot{r} U^- \dot{r}^{-1} \subset U^+ \cap (\dot{r} \dot{s_i})U^- (\dot{s_i}^{-1}\dot{r}^{-1})$, we conclude that $g_1' \in U^- \cap (\dot{r}\dot{s}_i U^- \dot{s}_i^{-1} \dot{r}^{-1})$. 

We have $g_2'x \in U^-g_2g_1x \subset \mathring{\CB}^r$ and $g_2'x  \in U^+x \subset \mathring{\CB}_w$.  Hence $g_2'x \in \mathring{\CB}_{r,w}$. Since $g_2' \in  U^+ \cap \dot{r} U^- \dot{r}^{-1}$, we can conclude that $g_2'x = c_{r,-}(x)$. Since  $g_1' \in U^- \cap (\dot{r}\dot{s}_i U^- \dot{s}_i^{-1} \dot{r}^{-1})$, we have $g_1'g_2'x = c_{rs_i,+} (g_2'x)$. This finishes the proof.
\end{proof}

We consider the rational function on $\mathring{\CB}_{e,w}$ by 
\[
\begin{tikzcd}
\mathring{\CB}_{e,w} \ar[r,dashed, "c_{rs_i,+}"]  \ar[d, dashed, "c_{r,-}"] & \mathring{\CB}_{e,rs_i} \ar[d, dashed, "c_{r,-}"] \\
\mathring{\CB}_{r,w} \ar[r,dashed,"c_{rs_i,+}"] & \mathring{\CB}_{r,rs_i} \ar[r, "\Delta^{rs_i\omega_i}_{r\omega_i}"]& \BC^\x. 
\end{tikzcd}
\]
By Proposition~\ref{AtlasRegular1}, we can extend the map to  
\[
T_{r,rs_i}:\mathring{\CB}_{e,w}\cap \dot{r}U^-B^+/B^+\to \mathring{\CB}_{r,w} \xrightarrow{D_{r,\mathbf{rs_i}}} \BC.
\]
Here the reduced expression for $rs_i$ is always chosen such that the last term is equal to $s_i$ to be consistent with Proposition~\ref{AtlasRegular1}.

\begin{cor}
\label{AtlasRegular2}
We have $T_{r,rs_i}=\Delta^{rs_i\omega_i}_{\omega_i}/\Delta^{r\omega_i}_{\omega_i}$ on $\mathring{\CB}_{e,w}\cap \dot{r}U^-B^+/B^+$, or as rational functions on $\mathring{\CB}_{e,w}$.
\end{cor}

\begin{proof}
The projection $\mathring{\CB}_{e,w}\cap \dot{r}U^-B^+/B^+ \to \mathring{\CB}_{r,w}$ is given by $g\dot{r}B^+\mapsto \sigma_{r,-}(g)\dot{r}B^+$. By definition, $g=g'\sigma_{r,-}(g)$ for some $g'\in \dot{r}U^-\dot{r}^{-1}\cap U^+$, so $\dot{r}^{-1}g' \dot{r}\in U^-$. By the assumption $r<rs_i$, $(\dot{r}\dot{s})^{-1}g' (\dot{r}\dot{s})$ also lies in $U^-$.

We compute that (the ration is well-defined by \eqref{eq:well})
\[
\frac{\Delta^{rs_i\omega_i}_{\omega_i}(g\dot{r})}{\Delta^{r\omega_i}_{\omega_i}(g\dot{r})}=\frac{\Delta^{rs_i\omega_i}_{\omega_i}(\sigma_{r,-}(g)\dot{r})}{\Delta^{r\omega_i}_{\omega_i}(\sigma_{r,-}(g)\dot{r})}=\Delta^{rs_i\omega_i}_{r\omega_i}(\sigma_{r,-}(g)) = D_{r, \mathbf{rs_i}}(\sigma_{r,-}(g)).
\]
The last equality follows from Proposition \ref{AtlasRegular1}. The corollary is proved.
\end{proof}

Corollary \ref{AtlasRegular2} shows that we can extend $T_{r,rs_i}$ to a rational function on $\dot{r}U^-B^+/B^+$ using the equivalent definition $\Delta^{rs_i\omega_i}_{\omega_i}/\Delta^{r\omega_i}_{\omega_i}$.
The same argument also shows that $T_{r,rs_i}$ interacts well with (\ref{AtlasMap2}) under a right weak Bruhat order assumption.
 
\begin{lem}
\label{AtlasRegular3}
Let $r<rs_i\leq_R v\leq w$ and $v'\leq_R r<rs_i\leq w$. Then $T_{r,rs_i}$ is preserved under the rational maps $c_{v,+}:\mathring{\CB}_{e,w}\dashrightarrow \mathring{\CB}_{e,v}$ and $c_{v',-}: \mathring{\CB}_{e,w}\dashrightarrow \mathring{\CB}_{v',w}$.
\end{lem}

\begin{proof}
We have two commutative diagrams by the definitions of $c_{?,-}$ and $c_{?, +}$:
\[
\begin{tikzcd}
\mathring{\CB}_{e,w} \ar[r, dashed, "c_{v,+}"] \ar[dr, dashed, "c_{rs_i,+}" left] & \mathring{\CB}_{e,v} \ar[d, dashed, "c_{rs_i,+}" ]\\
 & \mathring{\CB}_{e,rs_i};
\end{tikzcd}
\qquad 
\begin{tikzcd}
\mathring{\CB}_{e,w}\ar[dr, dashed, "c_{r,-}" left]  \ar[r, dashed,"c_{v',-}" above] & \mathring{\CB}_{v',w} \ar[d, dashed,"c_{r,-}"]\\
 & \mathring{\CB}_{r,w}.
\end{tikzcd}
\]

The following lemma follows immediately from Lemma~\ref{lem:rrs}.
\end{proof}

\subsection{The coordinate ring \texorpdfstring{$\BC[\mathring{\CB}_{v,w}]$}{C[Bvw]}}
\label{sec:RichardsonRing}
 In this subsection, we establish various properties on the coordinate ring  $\BC[\mathring{\CB}_{v,w}] $ for $v\le w$.

\begin{prop}\label{prop:CB}
We fix a reduced expression $\bfv = (i_1, \dots, i_m) $ of $v$. We define $v_{(l)} = v_{i_1} \dots v_{i_{l}}$. Then we have 
\[
 \BC[\mathring{\CB}_{v,w}] \cong \BC[\mathring{\CB}_{e,w}][T^{\pm 1}_{v_{(l-1)}, v_{(l)}}; 1 \le l \le m]/(T_{v_{(l-1)}, v_{(l)}} =1 ; 1 \le l \le m).
\]
\end{prop}
\begin{proof}
Recall the atlas isomorphism 
\[
\mathring{\CB}_{e,w}\cap \dot{v}_{(1)}U^-B^+/B^+ \isoto\mathring{\CB}_{e,v_{(1)}}\x \mathring{\CB}_{v_{(1)},w}.
\]
The closed subvariety defined by $T_{e,v_{(1)}} =1$ is isomorphisc to 
\[
\{\ast\}\x \mathring{\CB}_{v_{(1)},w} \cong \mathring{\CB}_{v_{(1)},w}.
\]
Since $T_{e,v_{(1)}}$ is regular on $\mathring{\CB}_{e,w}$, we have 
\[
 \BC[\mathring{\CB}_{e,w} \cap (\dot{v}_{(1)}U^-B^+/B^+)] =  \BC[\mathring{\CB}_{e,w}][T_{e,v_{(1)}}^{-1}].
\]
Therefore, we have 
\[
 \BC[\mathring{\CB}_{e,w}][T_{e,v_{(1)}}^{\pm1}]/(T_{e,v_{(1)}}  =1) = \BC[\mathring{\CB}_{e,w}]/(T_{e,v_{(1)}} =1) \cong \BC[\mathring{\CB}_{v_{(1)},w}].
\]

By Proposition~\ref{AtlasRegular1}, we can then consider the regular function $T_{v_{(1)}, v_{(2)}} : \mathring{\CB}_{v_{(1)}, w} \rightarrow \mathring{\CB}_{v_{(1)}, v_{(2)}}$. Note that $v_{(l-1)} \le_R v_{(l)}$. By the commutative diagram in Lemma~\ref{AtlasRegular3}, we have 
\[
 \BC[\mathring{\CB}_{e,w}][T_{e,v_{(1)}}^{\pm1}, T_{v_{(1)}, v_{(2)}}^{\pm1}]/(T_{e,v_{(1)}} = T_{v_{(1)}, v_{(2)}}  =1) \cong 
 \BC[\mathring{\CB}_{v_{(1)},w}]/ (T_{v_{(1)}, v_{(2)}} = 1) \cong  \BC[\mathring{\CB}_{v_{(2)},w}].
\]
The proposition is proved by repeating the process.
\end{proof}

\begin{prop}\label{RichardsonFactorial}
$\BC[\mathring{\CB}_{v,w}]$ is a unique factorization domain for any $v\leq w$. 
\end{prop}

\begin{proof}
 We know $\BC[\mathring{\CB}_{e,w}]$ is a unique factorization domain, since it is the non-vanishing locus of a single equation on the affine space $B^+\dot{w}B^+/B^+ \cong \BA^{\ell(w)}$;
 see also \cite{SW21}*{Theorem~2.30}. 
 Let $v \le vs_i \le w$ for some $i \in I$. Assume $\BC[\mathring{\CB}_{v,w}]$ is a unique factorization domain, we show $\BC[\mathring{\CB}_{vs_i,w}]$ is also a unique factorization domain. The theorem would follow immediately.

Let $r = v s_i$. By the atlas isomorphism \eqref{AtlasMap2}, we know 
\[
\BC[\mathring{\CB}_{vs_i,w}][x^{\pm1}] \cong \BC[\mathring{\CB}_{vs_i,w}]\otimes_\BC\BC[\mathring{\CB}_{v,vs_i}] \cong \BC[\mathring{\CB}_{v,w}\cap (\dot{r}U^{-}B^+/B^+)].
\]
Since $\BC[\mathring{\CB}_{v,w}]$ is a unique factorization domain, the localization 
$\BC[\mathring{\CB}_{v,w}\cap (\dot{r}U^{-}B^+/B^+)] = \BC[\mathring{\CB}_{v,w}][T^{-1}_{v,r}]$ is still a unique factorization domain.
Hence $\BC[\mathring{\CB}_{vs_i,w}]$ is also a unique factorization domain.
\end{proof}

\section{Rational quasi-cluter maps on configuration spaces}

In this section, we introduce rational quasi-cluster maps between cluster algebras. We then construct various rational quasi-cluster maps between configuration spaces using results from \cite{SW21}. We then deduce crucial computational results to be used in \S\ref{sec:upper}.

\subsection{Rational quasi-cluster maps}\label{sec:subalg}

\begin{defi}
\label{SubalgDefi}
Let $(\CA,\bfs)$ and $(\CA',\bfs')$ be cluster algebras with initial seeds $\bfs$ and $\bfs'$. We say an embedding of fields of fractions $\sigma: K \rightarrow K'$ is rational quasi-cluster, denoted by
\[
\begin{tikzcd}
    (\CA,\bfs) \ar[r,dashed,"(\tilde{J} { , }S')"] & (\CA',\bfs'),
\end{tikzcd}
\]
if we have 
\begin{itemize}
    \item a set $\Tilde{J}\subset J$ containing $J_\uf$;
    \item a embedding of sets $\iota: \Tilde{J} \rightarrow J'$ mapping $J_\uf $ to $J'_\uf$;
    \item a subset $S' \subset J' - \iota(\tilde{J})$;
\end{itemize} 
satisfying the following properties:
\begin{enumerate}
    \item for any $i\in \Tilde{J}$, $\sigma(A_i)$ and $ A_{\iota(i)}'$ differ by a Laurent monomial in $A'_j$, $j\in S'$;
    \item for any $i\in J_\uf$, $\sigma(X_i) = X'_{\iota(i)}$. Here the exchange ratios are defined as
    \begin{equation}\label{eq:X}
    X_i=\prod_{j\in J} A_j^{\ve_{ij}} \quad \text{and} \quad X'_{\iota(i)}=\prod_{j\in J'} (A'_j)^{\ve_{\iota(i) j}}.
    \end{equation}
    \item for any $i\in J-\Tilde{J}$, $\sigma(A_i)$ is contained in the multiplicative group generated by $A'_i$, $i\in S'$.
\end{enumerate}
\end{defi}

\begin{remark}
The precise subset $S'$ is usually not important. We shall often omit it after establishing some basic properties in this subsection.
\end{remark}

\begin{remark}
One can consider rational quasi-cluster homomorphisms as a generalization of the quasi-cluster homomorphsims introduced by Fraser \cite{Fra16}. 
\end{remark}

The following lemma is immediate from the definition.

\begin{lem}
\label{SubalgTransitive}
Let $(\CA,\bfs)$, $(\CA',\bfs')$, and $(\CA'',\bfs'')$ be three cluster algebras. Assume that 
\begin{itemize}
    \item 
    \begin{tikzcd}
    (\CA,\bfs) \ar[r,dashed,"(\tilde{J} { , }S')"] & (\CA',\bfs'); 
\end{tikzcd}
    \item 
    \begin{tikzcd}
    (\CA',\bfs') \ar[r,dashed,"(\tilde{J'} { , }S'')"] & (\CA'',\bfs''); 
\end{tikzcd}
    \item $\iota(\Tilde{J})\subset \widetilde{J'}$.
\end{itemize}
Then we have 
\begin{tikzcd}
    (\CA,\bfs) \ar[rrr,dashed,"(\Tilde{J} { , }S''\cup \iota'(S'\cap \Tilde{J}'))"] &&& (\CA'',\bfs''). 
\end{tikzcd}
\end{lem}

The assumption on $X_i$ implicitly impose a relation on the exchange matrices.

\begin{lem}\label{lem:subalgex}
Let \begin{tikzcd}
    (\CA,\bfs) \ar[r,dashed,"(\tilde{J} { , }S')"] & (\CA',\bfs').
\end{tikzcd} 
\begin{itemize}
    \item If $i\in J_\uf$ and $j\in \Tilde{J}$, then $\ve_{ij}=\ve'_{\iota(i)\iota(j)}$.
    \item If $i\in J_\uf$ and $j\in J' - (S'\cup \iota(\Tilde{J}))$, then $\ve'_{\iota(i)j}= 0$.
\end{itemize}
\end{lem}

\begin{proof}
Recall $X_i=\prod_{j\in J} A_j^{\ve_{ij}}$ and $X'_{\iota(i)}=\prod_{j\in J'} (A'_j)^{\ve'_{\iota(i)j}}$. Then we compute 
\begin{align*}
&\sigma(X_i) =\prod_{j\in J} \sigma(A_j)^{\ve_{ij}} = \prod_{j\in \Tilde{J}} {A'}^{\ve_{ij}}_{\iota(j)} \prod_{k \in S'} {A'}_k^{c_k}, \text{ for some } c_k \in \BZ;   \\
&X'_{\iota(i)} =  \prod_{j\in J'} {A'_j}^{\ve'_{\iota(i)j}} = \prod_{j\in \iota(\tilde{J})}{A'_j}^{\ve'_{\iota(i)j}} \prod_{j\in S'} {A'_j}^{\ve'_{\iota(i)j}} \prod_{j\in J' - (S'\cup \iota(\Tilde{J}))} {A'_j}^{\ve'_{\iota(i)j}}. 
\end{align*}
Since the cluster variables in $\bfs'$ are algebraically independent, the identity $\sigma(X_i) =X'_{\iota(i)}$ implies the desired identities on the exponents. 
\end{proof}

One key property of rational quasi-cluster maps is that it is preserved under mutations. 

\begin{prop}
\label{SubalgMutation}
Let \begin{tikzcd}
    (\CA,\bfs) \ar[r,dashed,"(\tilde{J} { , }S')"] & (\CA',\bfs').
\end{tikzcd} Then for any $k\in J_\uf$, we have 
\[
\begin{tikzcd}
    (\CA,\mu_k(\bfs)) \ar[r,dashed,"(\tilde{J} { , }S')"] & (\CA',\mu_{\iota(k)}(\bfs')).
\end{tikzcd} 
\]
\end{prop}

\begin{proof} Here the embedding of fields $K \rightarrow K'$ and the embedding $\tilde{J} \rightarrow J'$ are the obvious ones. We check the three conditions in Definition~\ref{SubalgDefi}. We write $f\sim g$ if $f,g\in K'$ differ by a Laurent monomial in $A'_j$, $j\in S'$. Note that if $a \sim b $ and $f \sim g$, then we have $af \sim bf$.

We check Condition $(1)$. 
By definition, $\mu_k(A_i)=A_i$ for $i\neq k$ and
\[
\mu_k(A_k)=\prod_{j\in J}A_j^{[-\ve_{kj}]_+}\frac{1+X_k}{A_k}.
\]
So we have 
$\sigma(\mu_k(A_i))=\sigma(A_i)\sim   A_{\iota(i)} = \mu_{\iota(k)}(A'_{\iota(i)})$ for $i\neq k$. We then compute 
\begin{align*}
   \sigma(\mu_k(A_k)) &= \prod_{j\in J}\sigma(A_j)^{[-\ve_{kj}]_+}\frac{1+\sigma(X_k)}{\sigma(A_k)} \\
   &\sim \prod_{j\in \tilde{J}}{A'_{\iota(j)}}^{[-\ve_{kj}]_+}\frac{1+X'_{\iota(k)}}{A'_{\iota(k)}} = \prod_{j\in \tilde{J}}{A'_{\iota(j)}}^{[-\ve'_{\iota(k)\iota(j)}]_+}\frac{1+X'_{\iota(k)}}{A'_{\iota(k)}}.
\end{align*}
The last identity is by Lemma~\ref{lem:subalgex}. On the other hand, we have 
\begin{align*}
    \mu_{\iota(k)}(A'_{\iota(k)}) = \prod_{j\in J'} {A'_j}^{[-\ve_{\iota(k)j}]_+}\frac{1+X'_{\iota(k)}}{A_{\iota(k)}'}.
\end{align*}
It follows from Lemma~\ref{lem:subalgex} that 
\[
\sigma(\mu_k(A_k)) \sim   \mu_{\iota(k)}(A'_{\iota(k)}).
\]
We have verified the first condition.

We verify Condition (2) now. Recall that  $X_{k,\bfs}=X_{k,\mu_k(\bfs)}^{-1}$ and $X_{\iota(k),\bfs'}=X_{\iota(k),\mu_{\iota(k)}(\bfs')}^{-1}$ by definition. It follows immediately that 
\[
\sigma(X_{k,\mu_k(\bfs)}) = X'_{\iota(k),\mu_{\iota(k)}(\bfs')}.
\]
For (mutable) $i\neq k$, we have
\begin{align*}
X_{i,\mu_k(\bfs)}&=X_{i,\bfs}A_k^{-\ve_{ik}}\mu_k(A_k)^{-\ve_{ik}}\prod_{j\neq k}A_j^{[\ve_{ik}]_+\ve_{kj}+\ve_{ik}[-\ve_{kj}]_+}\\
&=X_{i,\bfs}(1+X_{k,\bfs})^{-\ve_{ik}}\prod_{j\neq k}A_j^{[\ve_{ik}]_+\ve_{kj}}\\
&=X_{i,\bfs}\left(1+X_{k,\bfs}^{-\sgn(\ve_{ik})}\right)^{-\ve_{ik}},  \text{ where $\sgn(\ve_{ik}) \in \{-1, 0, +1\}$}.
\end{align*}
Applying $\sigma$ and using Lemma \ref{lem:subalgex}, this shows that $\sigma(X_{i,\mu_k(\bfs)})=X'_{\iota(i),\mu_{\iota(k)}(\bfs')}$ for all $i\in J_\uf$.

Finally, Condition (3) is trivially preserved under mutations at $k\in J_\uf$ and $\iota(k)\in \iota(J_\uf)$.
\end{proof}

Upon mild modifications, we can extend \eqref{eq:X} to include frozen vertices as well. Let $(p_{ij})_{i,j\in J}$ be an auxiliary matrix with rational coefficients such that 
\begin{equation}\label{eq:au}
\begin{cases}
p_{ij} = 0, &\text{if } \{i,j\} \cap J_{uf} \neq \emptyset;\\
p_{ij} + \varepsilon_{ij,\bfs} \in \BZ, &\text{otherwise. }
\end{cases}
\end{equation}
Note that the choice of $p_{ij}$ is independent of the seed $\bfs$. Then for any seed $\bfs$ of $\CA$, the matrix $(\varepsilon_{ij,\bfs})_{i,j\in J} + (p_{ij})_{i,j\in J}$ has $\BZ$-coefficients, and satisfies the exchange relations. Then for any $i \in J$ (potentially frozen), we define the exchange ratio
\begin{equation}\label{eq:Xf}
 X_{i,\bfs} = \prod_{j\in J} A_{j,\bfs}^{\ve_{ij,\bfs} + p_{ij}}.
\end{equation}

We similarly define a rational matrix $(q_{ij})_{i,j\in J'}$, and the exchange ratio  $X'_{\iota(i),\bfs'}$, where $\iota(i)$ is potentially frozen. Here the definition of both $ X_{i,\bfs}$ and $X_{\iota(i),\bfs'}$ depend on the chosen matrix, which is usually clear from the context. 

The following corollary is straightforward from the proof of Proposition~\ref{SubalgMutation}.
 
\begin{cor}\label{cor:X_ifro}
Let \begin{tikzcd}
    \sigma: (\CA,\bfs) \ar[r,dashed,"(\tilde{J} { , }S')"] & (\CA',\bfs')
\end{tikzcd} and assume $\sigma(X_{i,\bfs}) = X'_{\iota(i),\bfs'}$. Then for any $k \in J_{\uf}$, we have 
\[
\sigma(X_{i,\mu_k(\bfs)}) = X'_{\iota(i),\mu_{\iota(k)}(\bfs')}.
\]
\end{cor}

\begin{rem}
All definitions and results in this section so far also apply to upper cluster algebras, since everything is considered purely inside the ambient fields and we only deal with cluster variables in a few seeds.
\end{rem}

\begin{defi}\label{defi:RationalV}
Let $V,V'$ be affine varieties. Assume $\BC(V) = K(\CA)$ for some (upper) cluster algebra $(\CA, \bfs)$, and $\BC(V') = K(\CA')$ for some (upper) cluster algebra $(\CA', \bfs')$.

A dominant rational map $\varphi:V'\dashrightarrow V$ is said to be \textit{rational quasi-cluster for $\bfs'$ and $\bfs$} if the induced map $K(\CA) \rightarrow K(\CA')$ is rational quasi-cluster in the sense of Definition~\ref{SubalgDefi}. We shall denote this by 
\[
\begin{tikzcd}
    (V',\bfs') \ar[r,dashed,"(\tilde{J} { , }S')"] & (V,\bfs).
\end{tikzcd}
\]

We often consider the same variety $V$ with several non-equivalent seeds, i.e., different cluster algebras in the same function field $\BC(V)$.
\end{defi}

\subsection{Rational quasi-cluster maps on configuration spaces}
\label{sec:constr:QCflag}

In this subsection, we define various rational maps between configuration spaces and show that they are rational quasi-cluster. For any signed expression $\bfw$, we fix the auxiliary matrix

\begin{equation}\label{eq:auLie}
 \begin{cases}
    p_{(i,k),(j,l)} =  \frac{1}{2}a_{ij}, &\text{if $i\neq j$ and $(k,l)=(0,0)$ or $(n_i,n_j)$;}\\
    p_{(i,k),(j,l)} = 0, &\text{otherwise. } 
 \end{cases}
\end{equation}

\subsubsection{The map $\pi_{\beta,\beta_2}$}

Let $\beta_1,\beta_2 \in \Br^+$, $\bfu=(i_1,\cdots,i_m)$ and $\bfv=(j_1,\cdots,j_n)$ be expressions for $\beta_1,\beta_2$ respectively. We then obtain an expression $\bfw$ for $\beta=\beta_1\beta_2$ via concatenation. 

Define $\pi_{\beta,\beta_2}: \Conf_\beta(\SA) \dashrightarrow \Conf_{\beta_2}(\SA)$ to be the dominant rational map sending the configuration (subject to additional conditions, since the map is only rational)
\begin{center}
\begin{tikzcd}
& & U^0\ar[dash,drr] & & \\
U_0 \ar[r,swap,"s_{i_1}"]\ar[dash,rru] & \cdots \ar[r,swap,"s_{i_m}"] &  U_m \ar[r,swap,"s_{j_1}"] & \cdots \ar[r,swap,"s_{j_n}"] & U_{m+n}
\end{tikzcd}
\end{center}
to
\begin{center}
\begin{tikzcd}
& U^0\ar[dash,dr] & \\
U_m \ar[r,swap,"s_{j_1}"]\ar[dash,ru] &  \cdots \ar[r,swap,"s_{j_n}"] & U_{m+n}
\end{tikzcd}
\end{center}
Note that $\pi_{\beta,\beta_2}$ is defined whenever $d(U^0,U_m)=e$, so it is regular on an open subspace of $\Conf_\beta(\SA)$.

We are primarily interested in the case when $\beta =w$, $\beta_1 = u$, $\beta_2 = v$ are all in the Weyl group $W \subset \Br^+$.

\begin{rem}
\label{ReductionRemark}
Let $w,w_1,w_2\in W \subset \Br^+$ with reduced expressions $\bfw,\bfw_1,\bfw_2$ be such that $\bfw$ is the concatenation of $\bfw_1$ and $\bfw_2$.

Under the identification $T \times \mathring{\CB}_{e,w} \cong \Conf_w(\SA)$, we compute $\pi_{w,w_2}$ as follows. In the configuration
\begin{center}
\begin{tikzcd}
& U^-\ar[dash,dr] & \\
tU^+ \ar[r,swap,"w_1"]\ar[dash,ru] & U_1   \ar[r,swap,"w_2"] & gB^+
\end{tikzcd}
\end{center}
When $B_1\in B^-B^+/B^+$, we can write $U_1=g_1t'U^+$ for some unique $g_1\in U^-$ and $t'\in T$ such that $(tU^+,g_1t'U^+)$ is compatible. Then for $g_2B^+=g_1^{-1}gB^+\in \mathring{\CB}_{e,w_2}$, we have $\pi_{w,w_2}(t,gB^+)=(t',g_2B^+).$
\end{rem}
Recall the seeds $\bfs(\bfw)$ and $\bfs(\bfv)$ for the cluster structures on $\Conf_\beta(\SA)$ and $\Conf_{\beta_2}(\SA)$ respectively. It follows from \S\ref{sec:prelim:Lie} that we have a natural embedding  $J(\bfv) \rightarrow J(\bfw)$ by constructing both seeds by amalgamation from the right.  

Let us be more specific. Assume each $s_i$ appeared $a_i$ times in $\bfu$. Then the embedding of the index set is given by 
\[
\iota:  J(\bfv) \rightarrow J(\bfw), \quad  (i,l)  \mapsto (i,l+a_i).
\]

\begin{prop}
\label{SubexpSubalg}
The map  
\[
\begin{tikzcd}
    \pi_{\beta,\beta_2}: (\Conf_\beta(\SA),\bfs(\bfw)) \ar[rr,dashed,"(J(\bfv) { , }\emptyset)"] && (\Conf_{\beta_2}(\SA),\bfs(\bfv))
\end{tikzcd}
\]is rational quasi-cluster.

Moreover, assume $i \in I$ appears $m_i > 0$ times in $\bfv$, and appears $n_i \ge m_i$ times in $\bfw$.  Then we have 
\[
\pi_{\beta,\beta_2}^*(A_{(i,m_i),\bfs(\bfv)})  = A_{(i, n_i),\bfs(\bfw)},\qquad  
\pi_{\beta,\beta_2}^*(X_{(i,m_i),\bfs(\bfv)})  = X_{(i, n_i),\bfs(\bfw)}.
\]
\end{prop}

\begin{proof}
Let $\bfs=\bfs(\bfv)$ and $\bfs'=\bfs(\bfw)$. From the construction of $J$ and $\ve$, it is clear that there is an injection $\iota:J \to J'$ sending $J_\uf$ to $J'_\uf$ such that for $i,j\in J$, $\ve_{ij}=\ve_{\iota(i)\iota(j)}$. From the picture above $\pi_{\beta,\beta_2}^*(A_i)=A'_{\iota(i)}$ for all $i\in J$, and $\pi_{\beta,\beta_2}^*(X_i)=X'_{\iota(i)}$ for all $i \in J_{\uf}$. We further have $\pi_{\beta,\beta_2}^*(X_{(i,m_i),\bfs(\bfv)})  =X_{(i, n_i),\bfs(\bfw)}$ by \eqref{eq:auLie}.
\end{proof}

\begin{remark}
It is important we consider the configuration space instead of the Richardson variety here. The extra torus factor guarantees that  $\sigma(X_i)=X'_{\iota(i)}$ for all $i \in J_\uf$, as well as the second claim. 

One can also consider an enlargement of $G$ like thickening \cite{BH24b} to achieve this. We use configuration spaces in order to use results in \cite{SW21} freely.
\end{remark}

\subsubsection{The map $\pi_R$}\label{sec:PiR}
We define $\pi_R:\Conf^{s_i}_\beta(\SA)\dashrightarrow\Conf_{\beta}(\SA)$ to be the dominant rational map
\[
\begin{tikzcd}
U^0\ar[r, "s_i"]\ar[dash,d] & U^1\ar[dash,d]\\
U_0\ar[r ] & U_n
\end{tikzcd}
\mapsto 
\begin{tikzcd}
U^0\ar[dash,d] \ar[dr, dash] &  \\
U_0\ar[r ] & U_n
\end{tikzcd}
\]
regular on the open subset given by $d(U^0,U_n)=e$.

\begin{prop}\label{prop:witow}
\begin{enumerate} 
\item The map
\[
\begin{tikzcd}\pi_R: (\Conf^{s_i}_\beta(\SA),\bfs(\bfw, \overline{i})) \ar[rr,dashed,"(J(\bfw) { , }\emptyset)"] && (\Conf_{\beta}(\SA),\bfs(\bfw)).
\end{tikzcd}
\]
is rational quasi-cluster.

\item Assume $i \in I$ appears $m_i > 0$ times in $\bfw$. Then we have 
\begin{align*}
\pi_{R}^*(X_{(i,m_i),\bfs(\bfw)}) &=  X_{(i, m_i),\bfs(\bfw, \overline{i})}  A_{(i, m_i+1),\bfs(\bfw,\overline{i})},\\
\pi_{R}^*(A_{(i,m_i),\bfs(\bfw)}) &=  A_{(i, m_i),\bfs(\bfw, \overline{i})}.
\end{align*}
\item Assume $j \neq i \in I$ appears $m_j > 0$ times in $\bfw$. Then we have 
\begin{align*}
\pi_{R}^*(X_{(j,m_j),\bfs(\bfw)}) &=  X_{(j, m_j),\bfs(\bfw, \overline{i})}A_{(i, m_i),\bfs(\bfw,\overline{i})}^{-a_{ji}}A_{(i, m_i+1),\bfs(\bfw,\overline{i})}^{a_{ji}},\\
\pi_{R}^*(A_{(j,m_j),\bfs(\bfw)}) &=  A_{(j, m_j),\bfs(\bfw, \overline{i})}
\end{align*}
\end{enumerate} 
\end{prop}
\begin{proof}
The proof is by direct computation similar to Proposition~\ref{SubexpSubalg}.
\end{proof}

\subsubsection{The map $R_i$}
Let $i\in I$. Define the reflection map $R_i:\Conf_{s_i\beta}(\SA)\to \Conf^{s_i}_\beta(\SA)$ as follows. For any configuration
\[
\begin{tikzcd}
& U^0 \ar[dash,dr] & \\
U_0 \ar[r,"s_i"]\ar[dash,ru] &  U_1 \ar[r,"\beta"] & U_n
\end{tikzcd}
\]
in $\Conf_{s_i\beta}(\SA)$, there is a unique compatible flag $U \in G/U^-$ such that $d(U,U^0)=s_i$ and $d(U,U_0)=s_i$. Then $d(U,U_1)=e$, so we get a unique configuration
\[
\begin{tikzcd}
U \ar[r,"s_i"]\ar[dash,d] & U^0\ar[dash,d]\\
U_1\ar[r,"\beta"] & U_n
\end{tikzcd}
\]
in $\Conf^{s_i}_\beta(\SA)$. It is easy to see that $R_i$ is an isomorphism by constructing an inverse similarly. The following proposition is essentially \cite{SW21}*{Proposition 4.11}.

\begin{prop} 
\label{ReflQuasiEquiv}
\begin{enumerate}
 \item The map 
\[
\begin{tikzcd}
    R_i: (\Conf_{s_i\beta}(\SA),\bfs(i,\bfw)) \ar[rrrr,dashed,"(J^+(\overline{i}{,}\bfw){,}J^0(\bfs(i{,}\bfw))"] &&&& (\Conf^{s_i}_\beta(\SA),\bfs(\overline{i},\bfw)).
\end{tikzcd}
\]
is rational quasi-cluster. 
\item For any $j \in I$ that appears $n_j > 0$ times in $(i,\bfw)$, we have $R^*_i (X'_{(j,n_j),\bfs(\overline{i},\bfw)}) = X_{(j,n_j),\bfs({i},\bfw)}$.
\item For any $j \in I$ that appears $n_j > 0$ times in $(i,\bfw)$, we have $R^*_i (A'_{(j,n_j),\bfs(\overline{i},\bfw)}) = A_{(j,n_j),\bfs({i},\bfw)}$.
\end{enumerate}
\end{prop}

\begin{proof}
Similar to the proof of Proposition \ref{SubexpSubalg}, we restrict to the open subspace of $\Conf_{s_i\beta}(\SA)$ consisting of configurations 
\[
\begin{tikzcd}
& U^-\ar[dash,dr] & \\
tU^+ \ar[r,"s_i"]\ar[dash,ru] &  B_1 \ar[r,"\beta"] & B_n
\end{tikzcd}
\]
such that $B_1 \in U^-B^+/B^+ \cap B^+\dot{s_i}B^+/B^+$. We write 
\[
B_1=y_i(a)B^+ = x_i(1/a)\dot{s}_iB^+=\dot{s}_iy_i(-1/a)B^+ \text{ for some } a\in \BC^\x.
\]

By direct calculation, we can rewrite the configuration as 
\[
\begin{tikzcd}
& U^-\ar[dash,dr] & \\
tU^+ \ar[r,"s_i"]\ar[dash,ru] &  \dot{s}_iy_i(-1/a)\dot{s}_i^{-1}t\dot{s}_iU^+ \ar[r,"\beta"] & B_n.
\end{tikzcd}
\]
We write $t'=\dot{s}_i^{-1}t\dot{s}_i$. Recall this implies that $t^{-1} \dot{s}_iy_i(-1/a) t' \in U^- \dot{s}_iU^+$. Then $R_i$ is given by
\begin{equation}\label{eq:iwi}
\begin{tikzcd} 
& U^-\ar[dash,dr] & \\
tU^+ \ar[r,"s_i"]\ar[dash,ru] &  y_i(a)B^+ \ar[r,"\beta"] & B_n
\end{tikzcd}
 \!\!\!\!\! \mapsto  \!\!\!\!\!
\begin{tikzcd}
\dot{s}_iU^-\ar[r]\ar[dash,d] & U^-\ar[dash,d]\\
\dot{s}_iy_i(-1/a)t'U^+\ar[r] & B_n
\end{tikzcd}
 = 
\begin{tikzcd}
U^-\ar[r]\ar[dash,d] &  y_i(1/a)\dot{s}_i^{-1}U^-\ar[dash,d]\\
 t'U^+\ar[r] &  y_i(1/a)\dot{s}_i^{-1}B_n.
\end{tikzcd}
\end{equation}

Let $\bfs=\bfs(i,\bfw)$ and $\bfs'=\bfs(\overline{i},\bfw)$. Then we have a natural identification of the index sets $J=J'$ and $J_\uf=J'_\uf$. We obtain the following identities for the exchange matrices from the construction 
\begin{align*}
&\ve_{ij}=\ve'_{ij}, \text{ for any }i\in J_\uf-\{(i,1)\}, j\in J;\\
&\ve_{(i,1)j}=\ve'_{(i,1)j}, \text{ for any }j \neq (i,0), (j,0);\\
&\ve_{(i,1)(i,0)}=1,\quad \ve'_{(i,1)(i,0)}=-1, \quad \ve_{(i,1)(j,0)}-\ve'_{(i,1)(j,0)}=a_{ij}.
\end{align*}
For the induced isomorphism $\sigma:K(\Conf^{s_i}_w(\SA))\to K(\Conf^e_{s_iw}(\SA))$, we have $\sigma(A'_k)=A_k$ for all $k\in J-\{(i,0)\}$ following the picture in \eqref{eq:iwi}. We also have $\sigma(X'_k)=X_k$ for all $k\in J_\uf^+-\{(i,1)\}$ from the exchange matrices. When $k\in J^+-\{(i,1)\}$ is frozen, we still have $\sigma(X'_k)=X_k$ by \eqref{eq:Xf} and \eqref{eq:auLie}.

We then compute
\begin{align*}
\frac{\sigma(X'_{(i,1)})}{X_{(i,1)}}&= \frac{\sigma(A'_{(i,0)})^{-1} \prod_{j\neq i}\sigma(A_{(j,0)})^{\ve'_{(i,1)(j,0)}}}{A_{(i,0)} \prod_{j\neq i}A_{(j,0)}^{\ve_{(i,1)(j,0)}}} 
 =\frac{\sigma(A'_{(i,0)})^{-1}}{A_{(i,0)}}\prod_{j\neq i}A_{(j,0)}^{-a_{ij}}\\
&=\frac{1}{\Delta_{\omega_i}(t')\Delta_{\omega_i}(t)}\prod_{j\neq i}\Delta_{\omega_j}(t)^{-a_{ij}} 
 =\frac{1}{\Delta_{\omega_i}(\dot{s}^{-1}_i t \dot{s}_i)\Delta_{\omega_i}(t)}\prod_{j\neq i}\Delta_{\omega_j}(t)^{-a_{ij}} \\
&=1.
\end{align*}
So $\sigma(X'_{(i,1)})=X_{(i,1)}$. In summary, we have $\sigma(X'_k)=X_k$ for all $k\in J^+$. This finishes the proof. 
\end{proof}

\subsection{Mutations and Reductions}
\label{sec:constr:mutate}

Let $i\in I$, $w\in W$, and $\bfw$ be a reduced expression for $w$.  Then we have two seeds $\bfs(\overline{i}, \bfw)$ and $\bfs( \bfw,\overline{i})$ for the configuration space $\Conf^{s_i}_\beta(\SA)$. Recall the sequence of mutations $\mu_{\vec{i}}$ in \eqref{eq:muiright}.

\begin{prop}\cite{SW21}*{Proposition~3.25}\label{prop:SW325}
   We have $\mu_{\vec{i}}(\bfs(\overline{i}, \bfw)) = \bfs( \bfw,\overline{i})$. We have the induced identity bijection $J(\overline{i}, \bfw) \leftrightarrow  J( \bfw, \overline{i})$, $(j,l) \mapsto (j,l)$.
\end{prop}

Assume $i$ appears $n_i >0$ times in $\mathbf{w}$. We write $\mathbf{w} = \mathbf{w}_1 \mathbf{s}_i \mathbf{w}_2$, where the support of $w_1$ does not contain $s_i$. In other words, we split $\mathbf{w}$ at $(i,1)$, the first appearance of $s_i$ from the left. Combining Proposition \ref{SubexpSubalg}, Proposition\ref{prop:witow} and Proposition\ref{ReflQuasiEquiv}, we get the following proposition.

\begin{prop}\label{prop:ww_2}
Assume each $s_j$ appeared $b_j$ times in $\bfw_1$. 
\begin{enumerate} 
\item The composition
\[
\begin{tikzcd}
    \Phi=R_i\circ\pi_{w,s_iw_2}: (\Conf_{w}(\SA),\mu_{\vec{i}}(\bfs(\bfw))) \ar[rrr,dashed,"(J^+(\bfw_2{,}\overline{i}) {\,,\,} S')"] &&& (\Conf^{s_i}_{w_2}(\SA),\bfs(\bfw_2{,}\overline{i})),
\end{tikzcd}
\]
is rational quasi-cluster. Here  $S'=\{k\in J(\bfw)\,|\, k<(i,1)\}$.

\item The composition $\varphi = \pi_R\circ R_i\circ\pi_{w,s_iw_2}$ is rational quasi-cluster with
\[
\begin{tikzcd}
    (\Conf_{w}(\SA),\mu_{\vec{i}}(\bfs(\bfw))) \ar[rrr,dashed,"(J^+(\bfw_2) {\,,\,} S')"] &&& (\Conf_{w_2}(\SA),\bfs(\bfw_2)).
\end{tikzcd}
\]
Here the embedding $J^+(\bfw_2) \rightarrow J(\bfw)$ is given by
\[
\begin{cases}
(j,l) \mapsto (j,l + b_j), &\text{if } j \neq i;
\\
(i,l) \mapsto (i,l + b_i ) =(i,l).
\end{cases}
\]
\end{enumerate} 
\end{prop}

\begin{remark}\label{rem:shift}
Note that there is secret shift on the index sets. Assume that each $j \in I$ appears $n_j \ge 0$ times in $\bfw$.  Then $i$ appears $m_i=n_i-1$ times in $\bfw_2$, and $j$ appears $m_j=n_j-a_j$ times in $\bfw_2$ if $j\neq i$. Therefore, for the last vertex $(j,m_j)$ in level $j$ in $J(\bfw_2)$, the embedding sends $(i,m_i)\mapsto(i,n_i-1)$ and $(j,m_j)\mapsto (j,n_j)$ for $j\neq i$.
\end{remark}

\begin{proof}
By Proposition~\ref{SubexpSubalg}, we first have 
\[
\begin{tikzcd}
    \pi_{w,s_iw_2}:(\Conf_{\bfw}(\SA),\bfs(\bfw)) \ar[rrr,dashed,"(J(i{,}\bfw_2) {\,,\,} \emptyset)"] &&& (\Conf_{\mathbf{s_i}\bfw_2}(\SA),\bfs(i,\bfw_2)).
\end{tikzcd}
\]
Then by Proposition~\ref{ReflQuasiEquiv}, we have 
\[
\begin{tikzcd}
   R_i:(\Conf_{s_iw_2}(\SA),\bfs(i,\bfw_2)) \ar[rrrr,dashed,"(J^+(\overline{i}{,} \bfw_2) {\,,\,} J^0(i{,} \bfw_2)"] &&&& (\Conf^{s_i}_{w_2}(\SA),\bfs( \overline{i},\bfw_2)).
\end{tikzcd}
\]
Then by Proposition~\ref{prop:SW325} and Proposition~\ref{SubalgMutation}, we have
\[
\begin{tikzcd}
   R_i:(\Conf_{s_iw_2}(\SA),\mu_{\vec{i}}(\bfs(i,\bfw_2))) \ar[rrrr,dashed,"(J^+(\bfw_2{,}\overline{i}) {\,,\,} J^0(i{,} \bfw_2)"] &&&& (\Conf^{s_i}_{w_2}(\SA),\bfs( \bfw_2,\overline{i})).
\end{tikzcd}
\]
Now the first statement follows from Lemma~\ref{SubalgTransitive} after making suitable adjustments to the various data $J_{\uf},S'$. Part (1) follows now. 

By Proposition~\ref{prop:witow}, we have
\[
\begin{tikzcd}
 \pi_R:(\Conf^{s_i}_{w_2}(\SA),\bfs( \bfw_2,\overline{i}))\ar[rrr,dashed,"(J( \bfw_2) {\,,\,} \emptyset)"] &&& (\Conf_{w_2}(\SA),\bfs( \bfw_2)),
\end{tikzcd}
\]
so Part (2) follows from the same argument.
\end{proof}
 
\subsection{The function $T_{r, rs_i}$ revisited}
Recall \S\ref{sec:T} that, for $r<rs_i \le w$, we defined a rational map 
\begin{equation}\label{eq:Trrs}
T_{r, rs_i} =  {\Delta^{rs_i \omega_i}_{\omega_i}}/{\Delta_{\omega_i}^{r \omega_i}} : \mathring{\CB}_{e,w} \dashrightarrow  \BC.
\end{equation}
We extend $T_{r,rs_i}$ as a rational function on $T \times \mathring{\CB}_{e,w} \cong \Conf_w(\SA)$ via the projection map.

Let $w \in W$ with a reduced expression $\bfw$. Assume $i$ appears $n_i >0$ times in $\bfw$. We split $\bfw = \bfw_1 s_i \bfw_2$ at the first appearance of $s_i$. Note that $i$ appears $n_i$ times in $(\bfw_2, \overline{i})$ as well (counting $\overline{i}$).

\begin{prop}
\label{OneStepMutation}
Let $\{A_k\}$ be the cluster variables in the seed $\bfs(\bfw_2,\overline{i})$ for $\Conf^{s_i}_{w_2}(\SA)$.
We have the following identity of rational functions on $\Conf_w(\SA)$
\[
T_{e,s_i}=(R_i\circ \pi_{w,s_iw_2})^*
\left(\frac{A_{(i,n_i-1)}}{A_{(i,n_i)}}\right).
\]
\end{prop}
 
\begin{proof}
We following the notations in Remark \ref{ReductionRemark}. Then $\pi_{w,s_iw_2}(t,gB^+)=(t',g_2B^+)$ for some $t'\in T$ and $g_2\in \mathring{\CB}_{e,s_iw_2}$.

By the proof of Proposition \ref{ReflQuasiEquiv}, $R_i\circ \pi_{w,s_iw_2}(t,gB^+)$ is the configuration (for some $t''\in T$)
\begin{center}
\begin{tikzcd}
& U^-\ar[dash,dr] & \\
tU^+ \ar[rr]\ar[dash,ru] & & gB^+
\end{tikzcd}
$\mapsto$
\begin{tikzcd}
\dot{s}_iU^-\ar[r]\ar[dash,d] & U^-\ar[dash,d]\\
\dot{s}_iy_i(-a^{-1})t''U^+\ar[r] & g_2B^+.
\end{tikzcd}
\end{center}

Recall in the seed $\bfs(\bfw_2,\overline{i})$ for $\Conf^{s_i}_{w_2}(\SA)$, the cluster variables $A_{(i,n_i-1)}$ and $A_{(i,n_i)}$ are given by $\Delta_{\omega_i}(U^0,U_1)$ and $\Delta_{\omega_i}(U^1,U_1)$, respectively, on the configuration
\begin{center}
\begin{tikzcd}
U^0\ar[r]\ar[dash,d]\ar[dash,dr] & U^1\ar[dash,d]\\
U_0\ar[r] & U_1
\end{tikzcd}
\end{center}

Let $g_2' \in B^-$ be the unique element such that $g_2'B^+=g_2B^+$ and $(\dot{s}_iy_i(-1/a)t''U^+, g'_2U^+)$ is compatible. Then $R_i\circ \pi_{w,s_iw_2}(t,gB^+)$ fits into the above picture with $U^0=\dot{s}_iU^-$, $U^1=U^-$, $U_0=\dot{s}_iy_i(-a^{-1})t''U^+$, $U_1=g_2'U^+$. We get
\[
(R_i\circ\pi_{w,s_iw_2})^*\left(\frac{A_{(i,n_i-1)}}{A_{(i,n_i)}}\right) (t, gB^+)=\frac{\Delta_{\omega_i}(\dot{s}_i^{-1}g'_2)}{\Delta_{\omega_i}(g'_2)}.
\]

Note that $\frac{\Delta_{\omega_i}(\dot{s}_i^{-1}g'_2)}{\Delta_{\omega_i}(g'_2)} = \frac{\Delta_{\omega_i}(\dot{s}_i^{-1}g_2)}{\Delta_{\omega_i}(g_2)} $ is independent of the choice of $g_2$ by \eqref{eq:well}. We further claim 
\[
\frac{\Delta_{\omega_i}(\dot{s}_i^{-1}g'_2)}{\Delta_{\omega_i}(g'_2)} = \frac{\Delta_{\omega_i}(\dot{s}_i^{-1}g_2)}{\Delta_{\omega_i}(g_2)} =\frac{\Delta_{\omega_i}(\dot{s}_i^{-1}g)}{\Delta_{\omega_i}(g)}.
 \]

We consider the open subspace of $T \times \mathring{\CB}_{e,w}$ such that $\pi_{w_{(k)}}^w(gB^+) \in U^-B^+/B^+$ for any $k \le (i,1)$. Let $\bfw_1=(i_1,\cdots,i_m)$. Then by \cite{MR04}*{Theorem~7.1}, we can write $g_1=y_{i_1}(b_1)\cdots y_{i_m}(b_m)$ for  $b_k\in \BC^\x$. Since $i_k\neq i$ for any $k$, we have $\dot{s}_i^{-1}y_{i_k}(a_k)\dot{s}_i\in U^-$ and $\dot{s}_i^{-1}g_1\dot{s}_i\in U^-$ as well. This shows the desired identity of rational functions.
\end{proof}

Suppose $r<s_ir<s_irs_j\leq w$ for $i, j \in I$. Let $\bfw = \bfw_1 \mathbf{s}_i \bfw_2$ be a reduced expression of $w$. Assume that the left-most subexpression for $s_irs_j$ in $\bfw$ is the same as the left-most subexpression for $s_irs_j$ in $s_i \bfw_2$. This means the $s_i$ above is the leftmost $s_i$ in $\bfw$ and $r<rs_j\leq w_2$. Since $r<rs_j\leq w_2$, the rational function $T_{r,rs_j}$ is well-defined on $\Conf_{w_2}(\SA)$.

Recall the rational map $\varphi=\pi_R\circ R_i\circ\pi_{w,s_iw_2}:\Conf_w(\SA)\dashrightarrow\Conf_{w_2}(\SA)$ defined in Proposition~\ref{prop:ww_2}.
\begin{lem}
\label{PullbackT}
We have  
\[
 \varphi^* (T_{r,rs_j}) = T_{s_ir,s_irs_j}.
\]
\end{lem}

\begin{proof}
We follow notations in Remark \ref{ReductionRemark}. By Proposition \ref{ReflQuasiEquiv}, under the identification $T \times \mathring{\CB}_{e,w} \cong \Conf_w(\SA)$, the map $\varphi$ is explicitly given by
\[
(t,gB^+)\mapsto (t',y_i(a^{-1})\dot{s}_i^{-1}g_2B^+), \quad \text{ for some $t'\in T$, $a\in \BC^\x$, $g_2\in \mathring{\CB}_{e,s_iw_2}$.}
\]
 Since $r<s_ir<s_irs_j$, we have $s_irs_j>rs_j$. So $\dot{s}_j^{-1}\dot{r}^{-1}y_i(a^{-1})\dot{r}\dot{s}_j\in U^-$.

Recall $g_1\in U^-$ defined in Remark \ref{ReductionRemark}. Then $g_1g_2B^+ = gB^+$. Since we are interested in the identity of rational functions, we further assume $g_1B^+\in \CR_{\bfe,\bfw}$ by \cite{MR04}*{Theorem~7.1}. So 
\[
g_1 = y_{i_1}(a_1) y_{i_2}(a_2) \cdots y_{i_k}(a_k), a_i \in  \BC^\x.
\]
Here $s_{i_1} \cdots s_{i_k} = \bfw_1$ is the reduced expression of $w_1$.  

The left-most subexpression assumption forces $s_{i_k}s_irs_j>s_irs_j$ for any $i_k$ appearing in $\bfw_1$.
Then $\dot{s}_j^{-1}\dot{r}^{-1}\dot{s}_i^{-1}y_{i_k}(b)\dot{s}_i\dot{r}\dot{s}_j\in U^-$  and $\dot{s}_j^{-1}\dot{r}^{-1}\dot{s}_i^{-1}g_1\dot{s}_i\dot{r}\dot{s}_j\in U^-$. Then we compute
\begin{align*}
\Delta^{rs_j \omega_j}_{\omega_j}(y_i(a^{-1})\dot{s}_i^{-1}g_2)&=\Delta_{\omega_j}(\dot{s}_j^{-1}\dot{r}^{-1}y_i(a^{-1})\dot{s}_i^{-1}g_2)\\
&=\Delta_{\omega_j}(\dot{s}_j^{-1}\dot{r}^{-1}\dot{s}_i^{-1}g_2)\\
&=\Delta_{\omega_j}(\dot{s}_j^{-1}\dot{r}^{-1}\dot{s}_i^{-1}g_1g_2)\\
&=\Delta^{s_irs_j \omega_j}_{\omega_j}(g_1g_2).
\end{align*}
Similarly, $\Delta^{r\omega_j}_{\omega_j}(y_i(a^{-1})\dot{s}_i^{-1}g_2)=\Delta^{s_ir \omega_j}_{\omega_j}(g_1g_2)$. We conclude that
\[
\varphi^* (T_{r,rs_j})(t,gB^+)
= 
\frac{\Delta^{rs_j \omega_j}_{\omega_j}(y_i(a^{-1})\dot{s}_i^{-1}g_2)}{\Delta^{r\omega_j}_{\omega_j}(y_i(a^{-1})\dot{s}_i^{-1}g_2)}
=
\frac{\Delta^{s_irs_j \omega_j}_{\omega_j}(g_1g_2)}{\Delta^{s_ir \omega_j}_{\omega_j}(g_1g_2)}
=T_{s_ir,s_irs_j}(t,gB^+).
\]
We finish the proof now.
\end{proof}

\section{Upper cluster structures on Richardson varieties}
\label{sec:upper}

In this section, we show that the open Richardson variety $\mathring{\CB}_{v,w}$ has an upper cluster structure. 

Our seed is constructed by consecutive freezing from the seed $\bfs(\bfw)$, and deleting a collection of isolated frozen vertices at the end. Geometrically, this corresponds to the realization of $\mathring{\CB}_{v,w}$ as a locally subvariety of $\mathring{\CB}_{e,w}$ following Proposition~\ref{prop:CB}. 

One can also construct the seed from $\bfs(\bfw)$ by consecutive freezing and deletion alternatively. Geometrically, this corresponds to the realization of 
$\mathring{\CB}_{v,w}$ as a closed subvariety of $\mathring{\CB}_{vs_i,w}$ ($vs_i \le v$) following Corollary \ref{cor:T1}, and then perform induction. There is no essential difference, except that the second approach would require us to formulate rather complicated induction hypotheses.

\subsection{Notations and seeds}\label{sec:notations}
We fix some notations to be used throughout this section.  

Let $v,w\in W$, $v\leq w$. Fix a reduced expression $\bfw=(i_1,\cdots,i_n)$ for $w$. Let $\bfv=(i'_1,\cdots,i'_n)$ be the left-most (recall \S\ref{sec:prelim:Weyl}) subexpression for $v$ in $\bfw$ and $1\leq k_1<\cdots <k_m\leq n$ be the indices where $i'_{k_l}=i_{k_l} (\neq e)$. In particular, we have  $m = \ell(v)$.
We further define 
\begin{itemize}
\item $\bfw^{(0)}=\bfw$, $w^{(0)}=w$, $\bfw^{(l)}=(i_{k_l+1},\cdots,i_n)$, and $w^{(l)}=s_{i_{k_l+1}}\cdots s_{i_n}$;
\item $v_{(l)}= s_{i_{k_1}} \cdots s_{i_{k_l}}$;  
\item $r_{(l)}= s_{i_{k_1}} \cdots s_{i_{k_{l-1}}}$ and $r_{(0)} = e$.
\end{itemize}

\begin{defi}\label{def:Ml}
Let $l\in\{1,\cdots,m\}$ and $i=i_{k_l}$. Suppose $i_{k_l}$ is the $a_l$-th appearance of $i$ in $\bfv$. Suppose in $\bfw$, there are $b_l$ appearances of $i$ before $i_{k_l}$ that is not in $\bfv$. In particular, $i_{k_l}$ is the $(a_l+b_l)$-th appearance of $i$ in $\bfw$.
\begin{enumerate}
\item We define $\Tilde{\mu}_l = \Tilde{\mu}_l(v, \bfw)$ to be the mutation sequence
\[
\Tilde{\mu}_l:=
\begin{cases}
\mu_{(i,n_i-a_l)}\circ\cdots \circ\mu_{(i,b_l+1)} & \text{if $a_l+b_l<n_i$};\\
\id & \text{if $a_l+b_l=n_i$}.
\end{cases}
\]
\item We also define
\[
M_l = M_l(v, \bfw)=\Tilde{\mu}_l\circ\cdots \Tilde{\mu}_1.
\]
\end{enumerate} 
Note that the definition depends on the reduced expression $\bfw$.
\end{defi}

\begin{defi}
\label{TildeSeedDefi}
We define a seed $\Tilde{\bfs}_l$ inductively starting from the seed $\Tilde{\bfs}_0 = \bfs(\bfw)$ for $\Conf_w(\SA)$ for $0 \le l \le m$. 
At step $l$,  suppose we are given the seed $\Tilde{\bfs}_{l-1}$. Mutate $\Tilde{\bfs}_{l-1}$ using the sequence $\Tilde{\mu}_l$. In $\Tilde{\mu}_l(\Tilde{\bfs}_{l-1})$, freeze all the vertices adjacent to $({i_{k_l}},n_{i_{k_l}}-a_l+1)$ to obtain the seed $\Tilde{\bfs}_l$. Note that the remaining vertices are unchanged. 

The natural cluster algebra $\CA(\Tilde{\bfs}_l)$ associated to the seed $\Tilde{\bfs}_l$ has the same field of fractions as $\CA(\bfs(\bfw))$.
\end{defi}

Note that we only freeze vertices, but keep the exchange matrices the same. In other words, we keep the arrows between frozen vertices. They are irrelevant to define the cluster algebra $\CA(\Tilde{\bfs}_l)$, but we need the additional data to define the exchange ratio of frozen vertices. This is crucial throughout this section. 

\begin{example}
We illustrate the definition by an example. Let 
\[
\bfw = s_2s_1s_2s_3s_4s_5s_4 s_2 s_3s_4 s_2 s_3 s_1 \in S_6.
\]
Let $v = s_2s_1s_3s_5 s_4s_2$. Then the left-most subexpression $\bfv$ is given by the underlined subexpressions 
\[
\underline{s_2} \,\underline{s_1}s_2 \underline{s_3}s_4 \underline{s_5} \,\underline{s_4}\, \underline{ s_2} s_3s_4 s_2 s_3 s_1
\] 

Then we have 
\[
\tilde{\mu}_1 = \mu_{(2,3)} \circ \mu_{(2,2)}\circ  \mu_{(2,1)},   
\quad \tilde{\mu}_2 = \mu_{(1,1)}, \quad \tilde{\mu}_3 = \mu_{(3,2)}\circ  \mu_{(3,1)}  
\]
\[
\tilde{\mu}_4 = id, \quad \tilde{\mu}_5 = \mu_{(4,2)}, 
\quad \tilde{\mu}_6 = \mu_{(2,2)}.
\]
Let us explain the computation for $\tilde{\mu}_5$. We are looking at the $5$th simple reflection (from the left) in $\bfv$, which is $s_4$. This is the first $s_4$ in $\bfv$. So $a_4 =1$. This corresponds to the $7$th simple reflection in $\bfw$. There are exactly one $s_4$ before in $\bfw$, that is not in $\bfv$, i.e., the $5$th element in $\bfw$. So $b_4 =1$.  We have $n_4 =3$, that is, there exactly three $s_4$ appearing in $\bfw$. This determines $\tilde{\mu}_5 = \mu_{(4,2)} = \mu_{(4,b+1)} = \mu_{(4,n_4 -a)}$.
\end{example}
 
\begin{remark}
This example is the same one considered in \cite{Men22}*{Example~B.1}. The algorithm in \cite{Men22} works only for finite types, while our algorithm works for any Kac--Moody types. Nevertheless, the algorithm in \cite{Men22} serves the motivation for us. A more conceptual (and new) reason will be given in the proof below.
\end{remark}

\subsection{The seed $\Tilde{\bfs}_l$}
\label{sec:upper:comp}

The main goal of this section is to establish various properties of the seed $\Tilde{\bfs}_l$ for $0 \le l \le m$.  

Consider the maps $\varphi_l:\Conf_{w^{(l-1)}}(\SA)\dashrightarrow\Conf_{w^{(l)}}(\SA)$ defined via Propositions~\ref{prop:ww_2} for any $1 \le l \le m$. Let
\[
\Tilde{\varphi}_{l-1}=\varphi_{l-1}\circ\cdots\circ \varphi_1:\Conf_w(\SA)\to \Conf_{w^{(l-1)}}(\SA).
\]

\begin{lem}\label{prop:tildephil}
The map  
\[
\begin{tikzcd}
    \Tilde{\varphi}_{l-1}: (\Conf_{w}(\SA),M_{l-1}(\bfs(\bfw))) \ar[rrr,dashed,"(J^+(\bfw^{(l-1)}) {\,,\,} S_{l-1})"] &&& (\Conf_{w^{(l-1)}}(\SA),\bfs(\bfw^{(l-1)})).
\end{tikzcd}
\]
is rational quasi-cluster. Here $S_{l-1}$ is some subset of $J(\bfs(\bfw))$ which we do not need to determine explicitly. 
\end{lem}
\begin{proof}

Let us illustrate the proof by following picture for $l=3$, which we hope would worth a thousand words. 
\[
\begin{tikzcd}
(\Conf_{w}(\SA),\bfs_0) \ar[r,dash, "\Tilde{\mu}_1 = \mu_{\overrightarrow{i_{k_1}}}"]\ar[rd, dashed] & (\Conf_{w}(\SA),M_1(\bfs_0)) \ar[d, dashed, "(J^+(\bfw^{(1)}){\,,\,} S_1)"]\ar[r,dash, "\Tilde{\mu}_2"] & (\Conf_{w}(\SA),M_2(\bfs_0)) \ar[d, dashed, "(J^+(\bfw^{(1)}){\,,\,} S_1)"]\\
& (\Conf_{w^{(1)}}(\SA),\bfs(\bfw^{(1)}))\ar[r,dash, "{\mu}_{\overrightarrow{i_{k_2}}}"] \ar[rd, dashed] & (\Conf_{w^{(1)}}(\SA),\mu_{\overrightarrow{i_{k_2}}}(\bfw^{(1)}))) \ar[d,dashed,"(J^+(\bfw^{(2)}) {\,,\,} S')"]\\
& & (\Conf_{w^{(2)}}(\SA),\bfs(\bfw^{(l)}))
\end{tikzcd}
\]
\begin{itemize}[leftmargin=*]
\item The horizontal lines are change of seeds and identities on the configuration spaces. 
\item The vertical arrows are $\tilde{\varphi}_1 = \varphi_1$ and $\tilde{\varphi}_2 = \varphi_2 \circ \varphi_1$, respectively. 
\item The set $S_1$ is given by Proposition~\ref{prop:ww_2}. 
\item The set $S'\subset J(\bfw^{(1)})$ is given by $S'= \{k\in J(\bfw^{(1)})\,|\, k< (i_{k_2},1)\}$.
\item Under the embedding $J^+(\bfw^{(1)})\to J(\bfw)$, $J^+(\bfw^{(1)})\cap S'$ maps to $\{k\in J(\bfw)\,|\, k_1<k<k_2\}$. 
Its union with $S_1$ is $S_2$. 
\item By Lemma \ref{SubalgTransitive}, we get the data $(J^+(\bfw^{(2)}), S_2)$ for $\Tilde{\varphi}_2$. 
\item All maps are rational quasi-cluster by Proposition~\ref{SubalgMutation}. 
\item The general case follows from extending this picture and induction.
\end{itemize}
Note that $\Tilde{\mu}_2 \neq \mu_{\overrightarrow{i_{k_2}}}$ for the seed $M_1(\bfs(\bfw))$ in general by Proposition~\ref{prop:ww_2} and Remark~\ref{rem:shift}. 
\end{proof}
Lemma~\ref{prop:tildephil} is our motivation for the definition of $\tilde{\mu}_l$ in Definition~\ref{def:Ml}. The following corollary is immediate now.
\begin{cor}
The map $\Tilde{\varphi}_{l-1}$ induces a rational quasi-cluster map
\[
\begin{tikzcd}
    (\Conf_{w}(\SA),\Tilde{\bfs}_{l-1}) \ar[rrr,dashed,"(J^+(\bfw^{(l-1)}) {\,,\,} S_{l-1})"] &&& (\Conf_{w^{(l-1)}}(\SA),\bfs(\bfw^{(l-1)})).
\end{tikzcd}
\]
\end{cor}
Note that on the left hand side, we consider the cluster algebra $\CA(\tilde{\bfs}_l)$ in the function field of $\Conf_{w}(\SA)$, as indicated by the seed; see Definition~\ref{defi:RationalV}.

Consider the commutative diagram of rational quasi-cluster maps 
\begin{equation}\label{eq:Phil}
\begin{tikzcd}
    (\Conf_{w}(\SA),\Tilde{\bfs}_{l-1}) \ar[rr,dashed,"\Tilde{\varphi}_{l-1}"] \ar[d,dashed, "\Phi_l"] && (\Conf_{w^{(l-1)}}(\SA),\bfs(\bfw^{(l-1)})) \ar[d,dashed,"\pi_{w^{(l-1)}, s_{i_{k_l}}w^{(l)}}" ]  \\ (\Conf^{s_{i_{k_l}}}_{ w^{(l)}}(\SA),\bfs(\overline{i_{k_l}},\bfw^{(l)}))&& (\Conf_{s_{i_{k_l}}w^{(l)}}(\SA),\bfs(i_{k_l},\bfw^{(l)})) \ar[ll,dashed,"R_i"].
\end{tikzcd}
\end{equation}
After mutations and applying Proposition~\ref{prop:SW325}, we hence have the rational quasi-cluster map
\[
\begin{tikzcd}
(\Conf_{w}(\SA), \Tilde{\mu}_l(\Tilde{\bfs}_{l-1})) \ar[r,dashed,"\Phi_l"] & (\Conf^{s_{i_{k_l}}}_{ w^{(l)}}(\SA),\bfs(\bfw^{(l)},\overline{i_{k_l}})).
\end{tikzcd}
\]
Let $\{A_i'\}$ be the cluster variables in $(\Conf^{s_{i_{k_l}}}_{ w^{(l)}}(\SA),\bfs(\bfw^{(l)},\overline{i_{k_l}}))$. Let $\{A_i\}$ be the cluster variables in $(\Conf_{w}(\SA), \Tilde{\mu}_l(\Tilde{\bfs}_{l-1}))$. Recall \eqref{eq:auLie} for the exchange ratios of frozen vertices. 
 
\begin{prop}\label{prop:XA}
Let $(i,m'_i)$ be the right-most (frozen) vertex in $\bfs(\bfw^{(l)},\overline{i_{k_l}})$ at level $i$ with $m'_i>0$, corresponding to the vertex $(i,m_i)$ in $\Tilde{\mu}_l(\Tilde{\bfs}_{l-1})$ under the rational quasi-cluster map $\Phi_l$.
\begin{enumerate}
\item $\Phi^*_l(A'_{(i,m'_i)})$ is a monomial of frozen variables.
\item $\Phi^*_l(X'_{(i,m'_i)}) = X_{(i,m_i)}$ times a monomial of frozen variables.
\end{enumerate} 
\end{prop}

\begin{proof}

We shall prove the proposition by induction on $l$. let us first consider the base case when $l=1$. The claims follow from Proposition~\ref{SubexpSubalg} and Proposition~\ref{ReflQuasiEquiv}.

 We consider the general case now. We write 
\[
\begin{tikzcd}
    (\Conf_{w}(\SA),\Tilde{\mu}_l(\Tilde{\bfs}_{l-1})) \ar[r,dashed,"\Phi_{l-1}"] \ar[dd,dashed, "\Phi_l"] 
    &
    (\Conf^{i_{k_{l-1}}}_{w^{(l-1)}}(\SA),\mu_{\overrightarrow{i_{k_l}}}(\bfs(\bfw^{(l-1)},\overline{i_{k_{l-1}}}))) \ar[d,dashed,"\pi_R" ]
    \\
    &
    (\Conf_{w^{(l-1)}}(\SA),\mu_{\overrightarrow{i_{k_l}}}(\bfs(\bfw^{(l-1)}))) \ar[d,dashed,"\pi_{w^{(l-1)}, s_{i_{k_l}}w^{(l)}}" ]  
    \\ 
    (\Conf^{s_{i_{k_l}}}_{ w^{(l)}}(\SA),\bfs(\bfw^{(l)},\overline{i_{k_l}}))
    & (\Conf_{s_{i_{k_l}}w^{(l)}}(\SA),\mu_{\overrightarrow{i_{k_l}}}(\bfs(i_{k_l},\bfw^{(l)}))) \ar[l,dashed,"R_i"] &.
\end{tikzcd}
\]

Let $j=i_{k_{l-1}}$. Denote by $\{A_i''\}$ the cluster variables for $(\Conf^{i_{k_{l-1}}}_{w^{(l-1)}}(\SA),\mu_{\overrightarrow{i_{k_l}}}(\bfs(\bfw^{(l-1)}, \overline{j})))$, and $(i,m''_i)$ be the right-most (frozen) index at level $i$. Under the rational quasi-cluster map $R_i\circ \pi_{w^{(l-1)}, s_{i_{k_l}}w^{(l)}}\circ \pi_R$, $(i,m'_i)$ corresponds to $(i,m''_i)$ in $J(\bfw^{(l-1)},\overline{j})$ for $i\neq j$, and $(j,m'_j)$ corresponds to $(j,m''_j-1)$.

Notice that by the choice of the auxiliary matrix, $X''_{(j,m''_j)}=A''_{(j,m''_j-1)}$. Meanwhile, $(j,m''_j)$ corresponds to the index $(j,n_j-a_{l-1}+1)$ in $\Tilde{\mu}_l(\Tilde{\bfs}_{l-1})$. By the definition of the seed $\Tilde{\bfs}_{l-1}$, all cluster variables appearing in $X_{(j,n_j-a_{l-1}+1)}$ are frozen variables. Therefore, by the induction hypothesis, $\Phi^*_{l-1}(X''_{j,m''_j})$ is a monomial of frozen variables.

When $i \neq j$, by Proposition~\ref{SubexpSubalg}, Proposition~\ref{prop:witow}, Proposition~\ref{ReflQuasiEquiv}, we compute that
\begin{align*}
\pi^*_R \circ \pi^*_{w^{(l-1)}, s_{i_{k_l}}w^{(l)}} \circ R_i^* (A'_{(i, m'_i)}) &= A''_{(i, m''_i)},\\
\pi^*_R \circ \pi^*_{w^{(l-1)}, s_{i_{k_l}}w^{(l)}} \circ R_i^* (X'_{(i, m'_i)}) &= X''_{(i, m''_i)}{A''}_{(j, m''_j-1)}^{-a_{ij}}{A''}_{(j, m''_j)}^{a_{ij}}\\
&= X''_{(i, m''_i)}{X''}_{(j, m''_j)}^{-a_{ij}}{A''}_{(j, m''_j)}^{a_{ij}}.
\end{align*}

Directly by the induction hypothesis, $\Phi^*_{l-1}(A''_{(i,m''_i)})$ is a monomial of frozen variables, while $\Phi^*_{l-1}(X''_{(i,m''_i)}) = X_{(i,m_i)}$ times a monomial of frozen variables. Therefore, $\Phi^*_l(A'_{(i,m'_i)})$ is a monomial of frozen variables, and $\Phi^*_l(X'_{(i,m'_i)}) = X_{(i,m_i)}$ times a monomial of frozen variables.

When $i=j$, we obtain similarly
\begin{align*}
\pi^*_R \circ \pi^*_{w^{(l-1)}, s_{i_{k_l}}w^{(l)}} \circ R_i^* (A'_{(j, m'_j)}) &= A''_{(j, m''_j-1)}\\
&= X''_{(j, m''_j)},\\
\pi^*_R \circ \pi^*_{w^{(l-1)}, s_{i_{k_l}}w^{(l)}} \circ R_i^* (X'_{(j, m'_j)}) &= X''_{(j, m''_j-1)}A''_{(j, m''_j)}.
\end{align*}

Notice that $(j,m''_j-1)$ is mutable in $\bfs(\bfw^{(l-1)}, \overline{j})$ and $\Phi^*_{l-1}$ is rational quasi-cluster, so we have $\Phi^*_{l-1}(X''_{(j, m''_j-1)}) = X_{(j,m_j)}$.  We also conclude that $\Phi^*_l(A'_{(j,m'_j)})$ is a monomial of frozen variables, and $\Phi^*_l(X'_{(j,m'_j)}) = X_{(j,m_j)}$ times a monomial of frozen variables.
\end{proof}

Recall the rational function $T_{?,?}$ defined in \eqref{eq:Trrs}. 

\begin{lem}\label{lem:tildephiT}
We have a commutative diagram of rational maps
\[
\begin{tikzcd}
    \Conf_{w}(\SA) \ar[d,dashed, "\Tilde{\varphi}_{l-1}"]\ar[rd, dashed, "T_{r_{(l)},v_{(l)}}"] &\\
    \Conf_{w^{(l-1)}}(\SA) \ar[r, dashed, "T_{e,s_{i_{k_l}}}" below]& \BC.
\end{tikzcd}
\]
\end{lem}
\begin{proof}
The lemma follows from Lemma~\ref{PullbackT}. Note that the assumptions in Lemma~\ref{PullbackT} are always  satisfied in each step, because $\bfv$ is the left-most subexpression for $v$ in $\bfw$.
\end{proof}

\begin{prop}
\label{TIsClusterMonomial}
Let $i = i_{k_l}$. We denote by $\{A_k\}$ the cluster variables in the seed $\Tilde{\bfs}_l$. Then we have the following identity of rational functions on $T\x \mathring{\CB}_{e,w} \cong \Conf_w(\SA)$: 
\[
T_{r_{(l)},v_{(l)}} =X_{(i,n_i-a_l+1)} \text{ times a monomial of frozen variables}.
\]

In particular, any mutable cluster variable $A_k$ appears in $T_{r_{(l)},v_{(l)}}$ if and only if $k$ is adjacent to $(i,n_i-a_l+1)$ in $\Tilde{\bfs}_l$.
\end{prop}

\begin{proof}

Denote by $A'_k$ the cluster variables in the seed $\bfs(\bfw^{(l)}, \overline{i})$. The right-most vertex $(i,m'_i)$ in $\bfs(\bfw^{(l)},\overline{i})$ corresponds to $(i,n_i-a_l+1)$ in $\Tilde{\mu}_l(\Tilde{\bfs}_{l-1})$ under the rational quasi-cluster map $\Phi_l$. By Proposition \ref{OneStepMutation}, $T_{e,s_i}$ on $\Conf_{w^{(l-1)}}(\SA)$ is given by
\begin{align*}
T_{e,s_i}&=(R_i\circ\pi_{w^{(l-1)},s_iw^{(l)}})^*\left(A_{(i,m'_i-1)}A^{-1}_{(i,m'_i)} \right)\\
&=(R_i\circ\pi_{w^{(l-1)},s_iw^{(l)}})^*\left( X_{(i,m'_i)}A^{-1}_{(i,m'_i)} \right).
\end{align*}

Now the proposition follows immediately from Proposition~\ref{prop:XA} and Lemma~\ref{lem:tildephiT}.
\end{proof}

\begin{remark}
It follows from the proof of Proposition~\ref{prop:XA} that we can actually obtain a rather precise expression for the monomial of frozen variables. We shall not need the precision here.
\end{remark}

\subsection{Richardson varieties}

Recall that a seed $\bfs(e, \bfw)$ for the cluster structure on $\mathring{\CB}_{e,w}$ can be obtained from the seed $\bfs(\bfw) $ for $\Conf_w(\SA)$ by deleting vertices indexed by $(i,0)$ and setting $A_{(i,0)}=1$ for $i\in I$.  

\begin{defi}
\label{RichardsonSeedDefi}
We define  seeds $\bfs_l$ inductively starting from the seed $\bfs_0 = \bfs(e, \bfw)$ for $\mathring{\CB}_{e,w}$ for $0 \le l \le m$. 
At step $l$, let $i=i_{k_l}$ and suppose we are given the seed $\bfs_{l-1}$. Mutate $\bfs_{l-1}$ using the sequence $\Tilde{\mu}_l$. In $\Tilde{\mu}_l(\bfs_{l-1})$, freeze all the vertices adjacent to $(i,n_i-a_l+1)$ then delete the vertex $(i,n_i-a_l+1)$ to obtain the seed $\bfs_l$. Note that the remaining cluster variables are unchanged.

Equivalently, the seed $\bfs_l$ can be obtained from the seed $\tilde{\bfs}_l$ by deleting the set of frozen vertices 
\[
\{(i,0) \,|\, i \in I\} \cup \{(i_{k_l}, n_{i_{k_l}} - a_l + 1) \,|\, 1 \le l \le m\}.
\]

We write $\bfs(v, \bfw) = \bfs_m$ for the seed obtained at the end of this process. Recall $m = \ell(v)$.
\end{defi}

\begin{remark}\label{rem:seedw}
The algorithm in Definition~\ref{RichardsonSeedDefi} does not require $\bfw$ to be reduced. In general, we obtain the seed for the cluster structure on the open Richardson variety on the twisted product of flag varieties; see Theorem~\ref{TwistedProductCluster}.
\end{remark}

\begin{thm}
\label{RichardsonCluster}
The coordinate ring of the open Richardson variety $\mathring{\CB}_{v,w}$ has an upper cluster algebra structure given by the seed $\bfs(v, \bfw)$. 
\end{thm}

\begin{proof}
We construct the following commutative diagram 
\[
\begin{tikzcd}
    \CO(\Conf_w(\CA))[T_{r_{(1)}, v_{(1)}}^{\pm 1}, \dots, T_{r_{(m)}, v_{(m)}}^{\pm1}] \ar[r, hook] \ar[d, ->>] & \ar[d, ->>] \CU(\tilde{\bfs}_m) = \CU(\bfs_m)[A^{-1}_{(i_{k_1}, n_{i_{k_1}} - a_1 + 1)}, \dots, A^{-1}_{(i_{k_m}, n_{i_{k_m}} - a_m + 1)}]  \\
    \BC[\CB_{v,w}] \ar[r, hook] & \CU(\bfs_m)
\end{tikzcd}
\]
\begin{itemize}[leftmargin=*]
       \item It follows by the construction of $\tilde{\bfs}_m$, the frozen vertices $\{(i_{k_s}, n_{i_{k_s}} - a_s + 1) \,|\, 1 \le s \le l\}$ are isolated.  We have $\CU(\tilde{\bfs}_m) = \CU(\bfs_m)[A^{-1}_{(i_{k_1}, n_{i_{k_1}} - a_1 + 1)}, \dots, A^{-1}_{(i_{k_m}, n_{i_{k_m}} - a_m + 1)}]$. 
       \item The map $\CO(\Conf_w(\CA))[T_{r_{(1)}, v_{(1)}}^{-1}, \dots, T_{r_{(m)}, v_{(m)}}^{-1}] \rightarrow \CU(\tilde{\bfs}_l)$ is induced by the standard cluster structure $\CO(\Conf_w(\CA)) \cong \CU(\bfs(\bfw))$ and repeatedly applying Proposition~\ref{FreezingInclusion} and Proposition~\ref{TIsClusterMonomial}.
         
        \item The map $\CO(\Conf_w(\CA))[T_{r_{(1)}, v_{(1)}}^{\pm 1}, \dots, T_{r_{(m)}, v_{(m)}}^{\pm1}] \rightarrow  \BC[\mathring{\CB}_{v,w}]$ is defined by 
        \[
         A_{(i,0)} \mapsto 1,\quad i \in I,
        \]
        \[
        T_{r_{(l)}, v_{(l)}} \mapsto 1,  \quad  1 \le l \le m.
        \]
        The quotient is indeed $ \BC[\mathring{\CB}_{v,w}]$ by Proposition~\ref{prop:CB}.
        \item
        The quotient map  $\CU(\tilde{\bfs}_m) \rightarrow \CU(\bfs_m)$ is defined by 
        \[
        A_{(i,0)} \mapsto 1, \quad  i \in I,
        \]
        \[
        A^{-1}_{(i_{k_l}, n_{i_{k_l}} - a_l + 1)} \mapsto \text{a frozen monomial determined by Proposition~\ref{TIsClusterMonomial} and $T_{r_{(l)}, v_{(l)}} \mapsto 1$}.
        \]
        \item The embedding $\BC[\CB_{v,w}]  \rightarrow   \CU(\bfs_m)$ is the natural induced map to make the diagram commute.
\end{itemize}
We show $\BC[\mathring{\CB}_{v,w}]  \rightarrow   \CU(\bfs_m)$ is surjective. By Proposition~\ref{TIsClusterMonomial} and Proposition~\ref{FreezingInclusion},  we have $\CA(\tilde{\bfs}_m ) \subset \CO(\Conf_w(\CA))[T_{r_{(1)}, v_{(1)}}^{\pm 1}, \dots, T_{r_{(m)}, v_{(m)}}^{\pm1}]$. Then we conclude that $\CA(\bfs_m) \subset \BC[\mathring{\CB}_{v,w}]$, that is,  all cluster variables are in $\BC[\mathring{\CB}_{v,w}]$.
By \cite{GLS13}*{Theorem 3.1}, all cluster variables are irreducible in $\CU(\bfs_m)$.  By Proposition \ref{RichardsonFactorial}, $\BC[\mathring{\CB}_{v,w}]$ is a UFD. By \cite{FWZ21}*{Corollary 6.4.6 and Remark 6.4.4}, we conclude that $\CU(\bfs_m)\subset\BC[\mathring{\CB}_{v,w}]$.
\end{proof}

\begin{remark}
By construction, in finite types, our seed coincides with the seed in \cite{Men22}, and \cite{CGG+} for open Richardson varieties. See \cite{CGG+}*{Section 10} for a comparison.
\end{remark}

In \cites{CGG+, GLSS2}, it was shown that in finite types, the cluster structure on any open Richardson variety is locally acyclic. We conjecture that this is still true for Kac--Moody types in general. In particular, this would imply the upper cluster algebra equals the cluster algebra.

\begin{conjecture}\label{conj}
For any $v\leq w$ in $W$ and reduced expression $\bfw$ for $w$, the seed $\bfs(v, \bfw)$ is locally acyclic.
\end{conjecture}

For the special case $w=vu$ with $\ell(w)=\ell(v)+\ell(v)$, $\CA=\CU$ can be shown directly by studying the seed $\bfs(v,\bfw)$ from Definition \ref{RichardsonSeedDefi}.

\begin{cor}
Assume $w = v u$ with $\ell(w) = \ell(v) + \ell(u)$. For any reduced expression $\bfw$ such that $\bfv = (i_1, i_2, \dots, i_m , e, \dots, e)$, we have $\CA(\bfs(v,\bfw)) = \CU(\bfs(v,\bfw))$.
\end{cor}

\begin{proof}
Let $\bfu = (i_{m+1}, \dots, i_n)$ be corresponding reduced expression of $u$. Under the assumption, the seed $\bfs(v,\bfw)$ is the same as the seed $\bfs(e,\bfu)$. The corollary follows from Proposition~\ref{IsolatedDeletion} and Theorem~\ref{dBSCluster}.
\end{proof}

\section{Twisted products of flag varieties}\label{sec:twisted}

In this section we extend the upper cluster structure on open Richardson varieties in flag varieties to open Richardson varieties in twisted products of flag varieties.

\subsection{The twisted product}\label{sec:twisted:defi}

The \textit{twisted product} of flag varieties is defined as
\[
\CZ_n=G\x^{B^+}G \x^{B^+} \cdots \x^{B^+} G/B^+, 
\]
with $n$ factors. We will sometimes omit the subscript $n$ if the context is clear.

The isomorphism $G^n\isoto G^n$, $(g_1,g_2,\cdots,g_n)\mapsto (g_1,g_1g_2,\cdots, g_1g_2\cdots g_n)$, induces an isomorphism $\CZ_n\to \CB^n$,
\[
(g_1,g_2,\cdots,g_n)\mapsto (g_1B^+,g_1g_2B^+,\cdots, g_1g_2\cdots g_nB^+).
\]
In particular, we have the convolution product $m:\CZ_n\to \CB$,
\[
(g_1,g_2,\cdots,g_n)\mapsto g_1g_2\cdots g_nB^+.
\]

Let $\bfw=(i_1,\cdots,i_n)$ be an expression in $W$. The Schubert cell in $\CZ$ corresponding to $\bfw$ is defined as
\[
\mathring{\CZ}_\bfw = B^+\dot{s}_{i_1}B^+\x^{B^+} B^+\dot{s}_{i_2}B^+ \x^{B^+} \cdots \x^{B^+} B^+\dot{s}_{i_n}B^+/B^+.
\]
The opposite Schubert cell in $\CZ$ corresponding to $v\in W$ is defined as
\[
\mathring{\CZ}^v = \{x\in \CZ_n\,|\, m(x)\in \mathring{\CB}^v\}.
\]
The open Richardson variety in $\CZ$ corresponding to $v$ and $\bfw$ is defined as $\mathring{\CZ}_{v,\bfw}=\mathring{\CZ}_\bfw\cap \mathring{\CZ}^v$.

By \cite{BH24b}, the variety $\mathring{\CZ}_{v,\bfw}$ is non-empty if and only if $v\leq m_*(\bfw)$, in which case it is irreducible of dimension $\ell(\bfw)-\ell(v)$.

\begin{rem}\label{rem:twisted}
The definition of Schubert cells in twisted products is more general in \cite{BH24b}, in terms of a sequence $\underline{w}=(w_1,\cdots,w_n)$ of elements in $W$. By replacing each $w_i$ with its reduced expression, $\mathring{\CZ}_{v,\underline{w}}\subset \CZ_n$ is isomorphic to $\mathring{\CZ}_{v,\bfw}\subset\CZ_m$ for some $m\geq n$.

Since we only consider the locally closed subvariety $\mathring{\CZ}_{v,\bfw}$ instead of its closure in $\CZ_n$, the two definitions give isomorphic varieties. 
\end{rem}

\begin{rem}Special cases of $\mathring{\CZ}_{v,\bfw}$ include reduced double Bruhat cells, Bott-Samelson varieties, and braid varieties by \cite{BH24b}. 
\end{rem}

\subsection{The thickening map}\label{sec:twisted:thick}

We follow \cite{BH24b} to define the thickening group. Let $\Tilde{I}=I\sqcup\{\infty_1,\cdots,\infty_{n-1}\}$. We define $\Tilde{A}=(\Tilde{a}_{ij})_{i,j\in \Tilde{I}}$ as follows:
\begin{itemize}
    \item $\tilde{a}_{ij}=a_{ij}$ for $i,j\in I$
    \item $\tilde{a}_{i,\infty_l}=\tilde{a}_{\infty_l,i}=-2$ for $i\in I$ and $1\leq l\leq n-1$.
    \item $\tilde{a}_{\infty_l,\infty_l}=2$ for $1\leq l\leq n-1$ and $\tilde{a}_{\infty_l,\infty_{l'}}=-2$ for $1\leq l,l'\leq n-1$ with $l\neq l'$.
\end{itemize}

Let $\Tilde{G}$ be the minimal Kac--Moody group of simply connected type associated to $(\Tilde{I},\Tilde{A})$ and $\Tilde{W}$ be its Weyl group. We denote by ${\Tilde{\CB}}$ be the flag variety for $\Tilde{G}$, and denote the Schubert cells, the opposite Schubert cells, and the open Richardson varieties similarly to \S\ref{sec:prelim:flag}.

The Weyl group $W$ can be identified with the parabolic subgroup of $\Tilde{W}$ corresponding to $I \subset \Tilde{I}$. So we view $W$ as a subgroup of $\Tilde{W}$. The group $G$ can be identified with the standard Levi subgroup of $\Tilde{G}$ corresponding to $I$. So we also view $G$ as a subgroup of $\Tilde{G}$ as well.

Fix a pinning $(\Tilde{T},\Tilde{B}^+,\Tilde{B}^-,x_i,y_i;i\in \Tilde{I})$ of $\Tilde{G}$ compatible with the pinning of $G$. We define the thickening map 
\begin{align*}
th:(\BC^\x)^{n-1}\x G^n &\to \Tilde{G},\\
(a_1,\cdots,a_{n-1},g_1,\cdots,g_n)&\mapsto g_1y_{\infty_1}(a_1)g_2y_{\infty_2}(a_2)\cdots y_{\infty_{n-1}}(a_{n-1})g_n.
\end{align*}
For any expression $\bfw$ in $W$, let
\[
th(\bfw)=(i_1,\infty_1,i_2,\infty_2,\cdots,\infty_{n-1},i_n)
\]
Let $\Tilde{w}$ be the corresponding element in $\Tilde{W}$. Note that $th(\bfw)$ is a reduced expression for $\Tilde{w}$.

Any element in $\mathring{\CZ}_\bfw$ can be uniquely written as $(\dot{s}_{i_1}y_1,\cdots,\dot{s}_{i_n}y_n)$ for some $y_k\in U^-$. Therefore, the thickening map induces an injection $th:(\BC^\x)^{n-1}\x\mathring{\CZ}_\bfw\to \Tilde{\CB}$. Since $\dot{s}_{i_k}y_k\in B^+\dot{s}_{i_k}B^+$ for all $k$ and $y_{\infty_l}(a_l)\in B^+\dot{s}_{\infty_l}B^+$ for all $l$, the image lies in $\mathring{\Tilde{\CB}}_{\Tilde{w}}$.

\begin{prop}{\cite{BH24b}*{Proposition 3.2}}
\label{ThickFibration}
Let $\Tilde{P}_I^\pm$ be the parabolic subgroups of $G$ containing $\Tilde{B}^\pm$ corresponding to $I$. Let
\[
\CZ'=\Tilde{P}_I^- \Tilde{B}^+/\Tilde{B}^+\cap \Tilde{P}_I^+(\Tilde{B}^+\dot{s}_{\infty_1}\Tilde{B}^+)\Tilde{P}_I^+(\Tilde{B}^+\dot{s}_{\infty_2}\Tilde{B}^+)\cdots(\Tilde{B}^+\dot{s}_{\infty_{n-1}}\Tilde{B}^+)\Tilde{P}_I^+/\Tilde{B}^+.
\]
Then there is a locally trivial $(\BC^\x)^{n-1}$ fibration $\pi:\CZ'\to \CZ$ sending $\mathring{\Tilde{\CB}}_{v,\Tilde{w}}$ to $\mathring{\CZ}_{v,\bfw}$. Furthermore, $th:(\BC^\x)^{n-1}\x\mathring{\CZ}_\bfw\to \Tilde{\CB}$ is a section of $\pi$.
\end{prop}

\subsection{Cluster structures}\label{sec:twisted:cluster}

By Proposition \ref{ThickFibration}, there is an isomorphism $\mathring{\Tilde{\CB}}_{v,\Tilde{w}}\isoto (\BC^\x)^{n-1}\x\mathring{\CZ}_{v,\bfw}$. Projection to the first component gives regular functions $\Delta_1$, $\cdots$, $\Delta_{n-1}$ on $\mathring{\Tilde{\CB}}_{v,\Tilde{w}}$.

\begin{prop}
\label{FibrationCluster}
Let $\bfs(v, th(\bfw))$ be the seed for $\mathring{\Tilde{\CB}}_{v,\Tilde{w}}$ constructed in Theorem \ref{RichardsonCluster} using the reduced word $th(\bfw)$ for $\Tilde{w}$. For $k=1,\cdots,n-1$, we have 
\[
A_{(\infty_k,1)}=\frac{1}{\Delta_k}\prod_{l=1}^{k-1}\Delta_l^{q_{kl}}, \qquad \text{where } q_{kl} =-\langle \alpha_{\infty_l}^\vee, s_{i_{l+1}}s_{\infty_{l+1}}\cdots s_{i_k}s_{\infty_k}\omega_{\infty_k}\rangle.
\]
\end{prop}

\begin{proof}
Let $g\Tilde{B}^+ =z\dot{\Tilde{w}}\Tilde{B}^+\in \mathring{\Tilde{\CB}}_{v,\Tilde{w}} $ for some $z\in \Tilde{U}^+$. Let $\Tilde{w}_{(k)}=s_{i_1}s_{\infty_1}\cdots s_{i_k}s_{\infty_k}$.

(a) {\it We claim $A_{(\infty_k,1)}(gB^+) = \Delta^{\omega_{i_k}}_{\Tilde{w}_{(k)}\omega_{\infty_k}}(z)$}.

Recall that $A_{(\infty_k,1)}$ is obtained by restricting the cluster variable $A'_{(\infty_k,1)}$ in the seed $\bfs(e,th(\bfw))$ for $\mathring{\Tilde{\CB}}_{e,\Tilde{w}}$ via the atlas isomorphism.  By definition, we haves 
$A'_{(\infty_k,1)}(g' \Tilde{B}^+) = \Delta^{\omega_{i_k}}_{\Tilde{w}_{(k)}\omega_{\infty_k}}(z')$ for $g' \Tilde{B}^+ = z' \dot{\Tilde{w}}\Tilde{B}^+\in \mathring{\Tilde{\CB}}_{e,\Tilde{w}}$ with $z'\in \Tilde{U}^+$.

Recall  the atlas isomorphism
\[
\mathring{\Tilde{\CB}}_{e,\Tilde{w}}\cap \dot{v}\Tilde{U}^-\Tilde{B}^+/\Tilde{B}^+\to \mathring{\Tilde{\CB}}_{e,v}\x \mathring{\Tilde{\CB}}_{v,\Tilde{w}}, \quad h\dot{v}\Tilde{B}^+\mapsto (\sigma_{v,+}(h)\dot{v}\Tilde{B}^+,\sigma_{v,-}(h)\dot{v}\Tilde{B}^+), \quad \text{for } h\in \dot{v}\Tilde{U}^-\dot{v}^{-1}.
\]
Let $\sigma_{v,-}(h)\dot{v}\Tilde{B}^+ = g \Tilde{B}^+ = z\dot{\Tilde{w}}\Tilde{B}^+ $. We have  
\begin{align*}
A_{(\infty_k,1)} (\sigma_{v,-}(h)\dot{v}\Tilde{B}^+) &= A'_{(\infty_k,1)} (h'\sigma_{v,-}(h)\dot{v}\Tilde{B}^+), \quad \text{for some }h' \in    \dot{v}\Tilde{U}^-\dot{v}^{-1}\cap \Tilde{U}^+ \\
& = \Delta^{\omega_{i_k}}_{\Tilde{w}_{(k)}\omega_{\infty_k}}(h'z)\\
 & =  \Delta^{\omega_{i_k}}_{\Tilde{w}_{(k)}\omega_{\infty_k}}(z).
\end{align*}
This proves Claim (a).

Let $\bfv=(i_1',i_2',\cdots,i_n')$ be the right-most subexpression of $v$ in $\bfw$. By \cite{BH24b}*{Section 3.2}, the expression 
$i(\bfv)=(i_1',e,i_2',e,\cdots,e,i_n')$ is the right-most subexpression of $v$ in $th(\bfw)$. We write $g_k=\dot{s}_{i_k}$ if $i_k'=i_k$ and $g_k=y_{i_k}(b_k)$ with $b_k\in \BR_{>0}$ if $i_k'=e$. We define
\[
\Tilde{G}_{i(\bfv),th(\bfw), >0}=\{g_1y_{\infty_1}(a_1)\cdots y_{\infty_{n-1}}(a_{n-1})g_n \in \Tilde{G}\,|\, a_k,b_k\in \BR_{>0}\}
\]
By \cite{MR04}*{Theorem 11.3} (or its generalization to Kac--Moody groups), we have \[
\Tilde{G}_{i(\bfv),th(\bfw),>0} \cong \Tilde{\CB}_{i(\bfv),th(\bfw),>0}.
\]

Since $\Tilde{\CB}_{i(\bfv),th(\bfw),>0}$ is Zariski dense in $\mathring{\Tilde{\CB}}_{v,\Tilde{w}}$ by \cite{Deo85}*{Theorem 1.1}, it suffices to check equality on $\Tilde{\CB}_{i(\bfv),th(\bfw),>0}$.
So we can assume $g = g_1y_{\infty_1}(a_1)\cdots y_{\infty_{n-1}}(a_{n-1})g_n\in \Tilde{G}_{i(\bfv),th(\bfw), >0}$, where $g_k=\dot{s}_{i_k}$ if $i_k'=i_k$ and $g_k=y_{i_k}(b_k)$ with $b_k\in \BR_{>0}$ if $i_k'=e$.

For $l\leq k$, we define 
\[
p_{kl} =-\langle \alpha_{i_l}^\vee, s_{\infty_l}s_{i_{l+1}}s_{\infty_{l+1}}\cdots s_{i_k}s_{\infty_k}\omega_{\infty_k}\rangle, \quad 
q_{kl} =-\langle \alpha_{\infty_l}^\vee, s_{i_{l+1}}s_{\infty_{l+1}}\cdots s_{i_k}s_{\infty_k}\omega_{\infty_k}\rangle.
\]
We have $p_{kl},q_{kl}\leq 0$ as $th(\bfw)$ is reduced. By \cite{MR04}*{Section 7}, we have
\[
A_{(\infty_k,1)}(g\Tilde{B}^+)=\prod_{l=1,i'_l=e}^k b_l^{p_{kl}}\prod_{l=1}^k a_l^{q_{kl}}.
\]

  Recall $y_i(b)=\dot{s}_iy_i(-1/b)x_i(b)\alpha_i^\vee(b)$ for any $b\in \BR_{>0}$. Then we can write 
\[
g\Tilde{B}^+ = g_1y_{\infty_1}(a_1)\cdots y_{\infty_{n-1}}(a_{n-1})g_n \Tilde{B}^+=g'_1y_{\infty_1}(a'_1)\cdots y_{\infty_{n-1}}(a'_{n-1})g'_n \Tilde{B}^+. 
\]
Here $a'_k > 0$, $g'_k=\dot{s}_{i_k}$ if $i_k'=i_k$, and $g'_k=\dot{s}_{i_k}y_{i_k}(b_k')$ for some $b_k'\in \BR_{>0}$ if $i_k'=e$. Direct computation shows that 
\[
a'_k=a_k\prod_{l=1,i_l'=e}^k b_l^{r_{kl}}, \quad \text{where }r_{kl}=-\langle s_{i_k}s_{i_{k-1}}\cdots s_{i_{l+1}}\alpha_{i_l}^\vee, \alpha_{\infty_k}\rangle.
\]

(b) {\it We conclude that $\Delta_k(g\Tilde{B}^+) = a'_k$.}

Now the proposition is equivalent to the following claim. 

(c) {\it We claim $p_{kl}=\sum_{l\leq j\leq k}q_{kj}r_{jl}$ for all $1 \le l\leq k$.} 

We prove by induction on both $k$ and $k-l$. Notice that $p_{kl},q_{kl},r_{kl}$ only depend on the expression $(i_l,\cdots,i_k)$. So it suffices to prove the claim for $p_{k1}$. The base case follows from the fact that $p_{kk}=-2$, $q_{kk}=-1$, $r_{kk}=2$ for all $k$.

Assuming the identity for $p_{k,1}$, we want to show that $p_{k+1,1}=\sum_{l\leq j\leq k+1}q_{k+1,j}r_{j,1}$. Let $p',q',r'$ be the same coefficients obtained using the expression $(i_1,i_3,\cdots,i_{k+1})$. By definition, $q'_{kj}=q_{(k+1),(j+1)}$ for all $k\geq j$.

We have
\begin{align*}
p_{k+1,1}&=-\langle \alpha_{i_1}^\vee, s_{\infty_1}s_{i_2}s_{\infty_2}\cdots s_{i_{k+1}}s_{\infty_{k+1}}\omega_{\infty_{k+1}}\rangle\\
&=-\langle \alpha_{i_1}^\vee+2\alpha_{\infty_1}^\vee, s_{i_2}s_{\infty_2}\cdots s_{i_{k+1}}s_{\infty_{k+1}}\omega_{\infty_{k+1}}\rangle\\
&=2q_{k+1,1}-\langle \alpha_{i_1}^\vee, s_{i_2}s_{\infty_2}\cdots s_{i_{k+1}}s_{\infty_{k+1}}\omega_{\infty_{k+1}}\rangle\\
&=2q_{k+1,1}-\langle \alpha_{i_1}^\vee-a_{i_2,i_1}\alpha_{i_2}^\vee, s_{\infty_2}\cdots s_{i_{k+1}}s_{\infty_{k+1}}\omega_{\infty_{k+1}}\rangle\\
&=2q_{k+1,1}-a_{i_2,i_1}p_{k+1,2}+p'_{k,1},
\end{align*}
and
\begin{align*}
r_{j,1}&=-\langle s_{i_j}\cdots s_{i_2}\alpha_{i_1}^\vee, \alpha_{\infty_j}\rangle\\
&=-\langle s_{i_j}\cdots s_{i_3}(\alpha_{i_1}^\vee-a_{i_2,i_1}\alpha_{i_2}^\vee), \alpha_{\infty_j}\rangle\\
&=-a_{i_2,i_1}r_{j,2}+r'_{j-1,1}.
\end{align*}

We compute that
\begin{align*}
\sum_{1\leq j\leq k+1}q_{k+1,j}r_{j,1}&=2q_{k+1,1}+\sum_{2\leq j\leq k+1}q_{k+1,j}(-a_{i_2,i_1}r_{j,2}+r'_{j-1,1})\\
&=2q_{k+1,1}-a_{i_2,i_1}\sum_{2\leq j\leq k+1}q_{k+1,j}r_{j,2}+\sum_{2\leq j\leq k+1}q_{k+1,j}r'_{j-1,1}\\
&=2q_{k+1,1}-a_{i_2,i_1}\sum_{2\leq j\leq k+1}q_{k+1,j}r_{j,2}+\sum_{1\leq j\leq k}q'_{k,j}r'_{j,1}\\
&=2q_{k+1,1}-a_{i_2,i_1}p_{k+1,2}+p'_{k,1}.
\end{align*}
We finish the proof now.
\end{proof}

\begin{thm}
\label{TwistedProductCluster}
The coordinate ring of the open Richardson variety  $\mathring{\CZ}_{v,\bfw}$ in the twisted product $\CZ$ has an upper cluster algebra structure given by the seed $\bfs(v, \bfw)$.
\end{thm}

\begin{proof}
By Proposition \ref{ThickFibration}, we have 
\[
\BC[\mathring{\CZ}_{v,\bfw}]\simeq \BC[\mathring{\CB}_{v,\Tilde{w}}]/(\Delta_k-1)_{k=1}^{n-1}.
\]
By Proposition \ref{FibrationCluster}, this quotient is equivalent to deleting the frozen variables $A_{(\infty_k,1)}$ for $k=1,\cdots,n-1$ from the cluster structure on $\mathring{\CB}_{v,\Tilde{w}}$. Therefore by Proposition~\ref{IsolatedDeletion}, we get an upper cluster structure on $\BC[\mathring{\CZ}_{v,\bfw}]$. It is clear the seed is given by Definition~\ref{RichardsonSeedDefi}; see Remark~\ref{rem:seedw}.
\end{proof}

\begin{remark}
In the case of braid varieties (special cases of $\mathring{\CZ}_{v,\bfw}$), we expect our seed to be the same as the seed in \cite{CGG+}.
\end{remark}

It follows from the isomorphism $\mathring{\Tilde{\CB}}_{v,\Tilde{w}}\isoto (\BC^\x)^{n-1}\x\mathring{\CZ}_{v,\bfw}$ and Proposition~\ref{RichardsonFactorial} that $\BC[\mathring{\CZ}_{v,\bfw}]$ is also a unique factorization domain. Further note that the validity of Conjecture~\ref{conj} would imply the seed $\bfs(v,\bfw)$ on $\BC[\mathring{\CZ}_{v,\bfw}]$ is also locally acyclic.

\section{Further properties}\label{sec:extra}

\subsection{Independence of reduced expressions}
\label{sec:extra:indep}

The construction of the seed for $\mathring{\CB}_{v,w}$ in Definition \ref{RichardsonSeedDefi} depends on the initial choice of reduced expression $\bfw$ for $w$. For a different choice of reduced expression $\bfw'$, the seeds $\bfs(\bfw)$ and $\bfs(\bfw')$ are mutation equivalent up to a relabeling. We show that the mutation equivalence descends to the seeds for $\mathring{\CB}_{v,w}$ with $v \le w$.

In the proofs of this section, we will assume that $G$ is symmetric for simplicity. The symmetrizable case follows from standard folding arguments; see \cite{FG09}.

\begin{lem}
\label{lem:subquiver}
Let $\bfs$ be a skew-symmetric seed and $J'\subset J$ be a non-empty subset. Let $\bfs'$ be the seed obtained from $\bfs$ by deleting all vertices in $J-J'$. Suppose there exist a mutation sequence $\Tilde{\mu}$ on $\bfs'$ such that $\Tilde{\mu}(\bfs')=\bfs'$. Then the same mutation sequence $\Tilde{\mu}$ applied to $\bfs$ gives a seed $\Tilde{\mu}(\bfs)=\bfs$.
\end{lem}

\begin{proof}
Recall the exchange graph is the graph whose vertices are the seeds and whose edges connect seeds related by a single mutation \cite{FZ07}. By \cite{CKLP}*{Theorem 4.6}, the exchange graph only depends on the mutable part of the exchange matrix. Therefore, by viewing the vertices involved in $\Tilde{\mu}$ as mutable and the rest as frozen, $\Tilde{\mu}(\bfs')=\bfs'$ implies that the cluster variables of $\Tilde{\mu}(\bfs)$ and $\bfs$ agree. Note that by the setting in \cite{CKLP}, we lose information on the frozen part of the exchange matrix, i.e. the $\ve_{ij}$ where both $i,j$ are not involved in $\Tilde{\mu}$, so we cannot immediately conclude that $\Tilde{\mu}(\bfs)=\bfs$. However we know all cluster variables must coincide. The result then follows from \cite{GSV08}*{Theorem 4}, which states that the seed is uniquely determined by its cluster variables.
\end{proof}

\begin{prop}
\label{RichardsonIndepRedExp}
The seeds $\bfs(v,\bfw)$ and $\bfs(v,\bfw')$ are mutation equivalent up to relabeling.
\end{prop}
\begin{proof}
We may assume that $\bfw$ and $\bfw'$ differ by a single braid move, which must be $(i,j,i)\sim(j,i,j)$ or $(i,j) \sim (j,i)$ in the symmetric case. The second case is trivial. We assume it is the first case. We can write $\bfw=(\bfw_1,i,j,i,\bfw_2)$ and $\bfw'=(\bfw_1,j,i,j,\bfw_2)$ for some $i,j\in I$ with $a_{ij}=a_{ji}=-1$. We write $\bfs = \bfs(v,\bfw)$ and $\bfs' =\bfs(v,\bfw')$.

Assume the expression $(i,j,i)$ in $\bfw$ corresponds to $(i,a)$, $(j,b)$ and $(i,a+1)$ in $J(\bfw)$. Then the expression $(j,i,j)$ in $\bfw'$ corresponds to $(j,b)$, $(i,a)$ and $(j,b+1)$ in $J(\bfw')$. We have $\bfs(\bfw')$ equals to $\mu_{(i,a)}(\bfs(\bfw))$, after relabeling the vertices by $(i,a)\mapsto (j,b)$, $(j,b)\mapsto (j,b+1)$ and $(i,a+1)\mapsto (i,a)$.

The left-most subexpression $\bfv$ of $v$ also depends on $\bfw$. The only possibilities for $\bfv$ are $(\bfv_1, 1,1,1, \bfv_2)$, $(\bfv_1, i,1,1, \bfv_2)$, $(\bfv_1, 1,j,1,\bfv_2)$, $(\bfv_1, i,j,1, \bfv_2)$, $(\bfv_1, 1,j,i,\bfv_2)$ and $(\bfv_1, i,j,i, \bfv_2)$. Then the left-most subexpression $\bfv'$ of $v$ also depends on $\bfw'$ are $(\bfv_1, 1,1,1, \bfv_2)$, $(\bfv_1, 1,i,1, \bfv_2)$, $(\bfv_1, j,1,1,\bfv_2)$, $(\bfv_1, 1, i,j, \bfv_2)$, $(\bfv_1, j,i,1, \bfv_2)$ and $(\bfv_1, j,i,j, \bfv_2)$, respectively.

Recall seeds $\tilde{\bfs}_l$ defined in Definition~\ref{TildeSeedDefi}. We similarly define  seeds $\tilde{\bfs}'_l$ using the reduced expression $\bfw'$ (still with respect to the same $v$), where the relevant sequence of mutations will be denoted by $\tilde{\mu}'_l$. We have $\bfs = \bfs_m$ and $\bfs' = \bfs'_m$ in this notation.

For brevity, equality of seeds in this proof will always be understood to be up to relabeling. We start with $\tilde{\bfs}_{0} = \mu_{(i,a)}(\tilde{\bfs}'_{0})$ and compare $\tilde{\bfs}_{l}$ with $\tilde{\bfs}'_{l}$ for $1 \le l \le m$.  We shall freely use notations from \S\ref{sec:notations}.

{\it (a) For any $l$ such that $k_l < (i,a)$, we claim $\tilde{\bfs}_{l} = \mu_{(i,a)}(\tilde{\bfs}'_{l})$. }

We proceed inductively. Recall the rational quasi-cluster map $\Phi_l$ in \eqref{eq:Phil}. 

Since $k_l < (i,a)$, the relevant $(i,j,i)$ must be a subexpression of $\bfw^{(l)}$. Then we have the following commutative diagram of mutations 
 \[
 \begin{tikzcd}
    \bfs(\overline{i_k},\bfw^{(l)}) \ar[r, "\Tilde{\mu}_{\overrightarrow{i_{k_l}}}"] \ar[d] &\bfs(\bfw^{(l)},\overline{i_k})\ar[d]\\
    \bfs(\overline{i_k},{\bfw'}^{(l)}) \ar[r, "\Tilde{\mu}_{\overrightarrow{i_{k_l}}}" ] & \bfs({\bfw'}^{(l)},\overline{i_k})
\end{tikzcd}
\]
The vertical mutations corresponds to $(i,j,i) \sim (j,i,j)$. Thanks to Proposition~\ref{lem:subquiver}, we see that 
\[
 \tilde{\mu}_{l} (\tilde{\bfs}_{l-1}) = \mu_{(i,a)} \circ \tilde{\mu}'_{l} \circ \mu_{(i,a)}  (\tilde{\bfs}_{l-1}) = \mu_{(i,a)} \circ \tilde{\mu}'_{l} (\tilde{\bfs}'_{l-1}).
\]
After freezing, we obtain that $\tilde{\bfs}_{l} = \mu_{(i,a)} ( \tilde{\bfs}'_{l})$. This proves Claim (a).
 
{\it (b) Assume $\bfv \neq (\bfv_1, 1,1,1, \bfv_2)$, then we have $\bfs =\bfs'$.} 

Assume $k_l \in \{(i,a), (j,b), (i,a+1)\}$. By (a), we can assume $\tilde{\bfs}_{l-1} = \mu_{(i,a)}(\tilde{\bfs}'_{l-1})$. We divide into several cases based on the subexpression $\bfv$.

\begin{enumerate}
\item Assume $k_l = (i,a)$, and $\bfv = (\bfv_1, i,1,1, \bfv_2)$. Then $\bfv' = (\bfv_1, 1,i,1, \bfv_2)$. Then by the definition of $\tilde{\mu}'_l$, we see that 
\[
 \tilde{\mu}_l (\tilde{\bfs}_{l-1}) = \tilde{\mu}_l \circ \mu_{(i,a)} (\tilde{\bfs}'_{l-1}) = \tilde{\mu}'_l(\tilde{\bfs}'_{l-1}).
\]
After freezing, we obtain that $\tilde{\bfs}_l = \tilde{\bfs}'_l$.
\item Assume $k_l = (j,b)$, and $\bfv = (\bfv_1, 1,j,1, \bfv_2)$. Then similar to case (1), we see that $\bfs_l = \bfs'_l$.
\item Assume $k_l = (i,a)$, and $\bfv = (\bfv_1, i,j,1, \bfv_2)$. Then we have $k_{l+1} = (j,b)$. Similar to the argument in case (1), we obtain that $\bfs_{l+1} = \bfs'_{l+1}$.
\item Assume $k_l = (j,b)$, and $\bfv = (\bfv_1, 1,j,i, \bfv_2)$. Then we have $k_{l+1} = (i,a+1)$. Similar to the argument in case (1) or (3), we obtain that $\bfs_{l+1} = \bfs'_{l+1}$.
\item Assume $k_l = (i,a)$, and $\bfv = (\bfv_1, i,j,i, \bfv_2)$. Note that $\bfw^{(l+2)} = (\bfw')^{(l+2)} = \bfw_2$. We claim $\tilde{\bfs}_{l+2} = \tilde{\bfs}'_{l+2}$.

We have the following commutative diagram of mutations
\[
 \begin{tikzcd}
    \bfs(\overline{i},\overline{j}, \overline{i}, \bfw_2) \ar[r] \ar[d] &\bfs(\bfw_2,\overline{i},\overline{j}, \overline{i})\ar[d]\\
    \bfs(\overline{j},\overline{i}, \overline{j},\bfw_2) \ar[r] & \bfs(\bfw_2,\overline{j},\overline{i}, \overline{j}).
\end{tikzcd}
\]
Thanks to Proposition~\ref{lem:subquiver} again, we have a commutative diagram of mutations on the seed $\tilde{\bfs}_l$. We conclude that 
\[
 \tilde{\mu}_{l+2} \circ \tilde{\mu}_{l+1} \circ \tilde{\mu}_{l} (\tilde{\bfs}_l) = {}'\mu \circ  \tilde{\mu}'_{l+2} \circ \tilde{\mu}'_{l+1} \circ \tilde{\mu}'_{l} (\tilde{\bfs}'_l). 
\]
Here ${}'\mu$ is the mutation corresponding to $(\overline{i},\overline{j}, \overline{i})\sim (\overline{j},\overline{i}, \overline{j})$. 

So $\tilde{\bfs}_{l+2}$ and $\tilde{\bfs}'_{l+2}$ only differ by three frozen vertices that will be deleted to obtain $\bfs$ and $\bfs'$, respectively.
\end{enumerate}

Summarizing the five cases above and continuing the construction following Definition~\ref{RichardsonSeedDefi}, we obtain that $\bfs =\bfs'$. This proves Claim (b).

{\it (c) Assume $\bfv = (\bfv_1, 1,1,1, \bfv_2)$, then we have $\bfs =\mu_{(i, a)}(\bfs')$.}

Assume $k_{l-1} < (i,a)$ and $k_l > (i, a+1)$. Then by Claim (a), we see that $\bfs_{l-1} =\mu_{(i, a)}(\bfs'_{l-1})$. Recall the rational quasi-cluster maps in \eqref{eq:Phil}
\[
\begin{tikzcd}
    (\Conf_{w}(\SA),\Tilde{\bfs}_{l-1}) \ar[rr,dashed,"\Tilde{\varphi}_{l-1}"]  && (\Conf_{w^{(l-1)}}(\SA),\bfs(\bfw^{(l-1)})), 
\end{tikzcd}
\]
\[
\begin{tikzcd}
    (\Conf_{w}(\SA),\Tilde{\bfs}'_{l-1}) \ar[rr,dashed,"\Tilde{\varphi}_{l-1}"]  && (\Conf_{w^{(l-1)}}(\SA),\bfs(\bfw^{(l-1)})). 
\end{tikzcd}
\]

For the seed $\bfs(\bfw^{(l-1)})$, the sequence $\mu_{\overrightarrow{i_{k_l}}}$ of mutations commute with the mutation for $(i,j,i) \sim (j,i,j)$. This is because the vertex corresponds to the first $i$ in $(i,j,i)$ is isolated from the any vertex involved in $\mu_{\overrightarrow{i_{k_l}}}$. By Proposition~\ref{lem:subquiver}, the sequence $\tilde{\mu}_{l}$ and the mutation $\mu_{(i,a)}$ commute for both $\Tilde{\bfs}_{l-1}$ and $\Tilde{\bfs}'_{l-1}$. We conclude that  $\bfs_{l} =\mu_{(i, a)}(\bfs'_{l})$. Similar argument shows $\bfs_{m} =\mu_{(i, a)}(\bfs'_{m})$. The claim follows now.

The proposition is proved by combining the three parts.
\end{proof}

The generalization to the twisted product is straightforward.  Let $\bfw$ and $\bfw'$ be two expressions (not necessarily reduced) differ by braid moves. So they represent the same element in $\Br^+_W$, and we have $\mathring{\CZ}_{v, \bfw} \cong \mathring{\CZ}_{v, \bfw'}$. Let $v \le m_*(\bfw) =m_*(\bfw')$. 
\begin{cor}
The seeds $\bfs(v,\bfw)$ and $\bfs(v,\bfw')$ are mutation equivalent up to relabeling.
\end{cor}

\subsection{Comparison of total positivity structures}
\label{sec:extra:TP}

We first recall the totally nonnegative part of flag varieties defined in \cites{Lus94, BH24a}. Let $U^-_{\geq0}$ be the submonoid of $G$ generated by $y_i(a)$ for $i\in I$ and $a\in \BR_{>0}$. Let $\CB_{\geq0}=\overline{U^-_{\geq0}B^+/B^+}$ be the closure of $U^-_{\geq0}B^+/B^+$ in $\CB$ with respect to the Hausdorff topology. For any $v\leq w$, let $\CB_{v,w,>0}=\mathring{\CB}_{v,w}\cap \CB_{\geq0}$.

Meanwhile, using the upper cluster structure, we can define another totally positive subset of $\mathring{\CB}_{v,w}$. For any seed $\bfs$, let $\CB_{v,w,>0}^{cl}$ be the cluster positive subset where all cluster variables in $\bfs$ take values in $\BR_{>0}$. Since the mutation relations are subtraction free, this definition is independent of the choice of seed $\bfs$ up to mutation equivalence (and relabeling).

\begin{prop}
\label{TPAgree}
$\CB_{v,w,>0}=\CB_{v,w,>0}^{cl}$.
\end{prop}
\begin{proof}
Retain the notations in \S\ref{sec:notations}. We prove by induction on $l$ that $\CB_{v_{(l)},w,>0}=\CB_{v_{(l)},w,>0}^{cl}$. The base case is $\CB_{e,w,>0}=\CB_{e,w, >0}^{cl}$, which follows by Lemma \ref{SchubertClusterVariable} and \cite{MR04}*{Theorem 12.1} (or \cite{BH21}).
 
By \cite{BH24a}*{Theorem 5.3} and our induction hypothesis, the rational map $
\mathring{\CB}_{v_{(l-1)},w}\dashrightarrow \mathring{\CB}_{v_{(l-1)},v_{(l)}}\x \mathring{\CB}_{v_{(l)},w}$ restricts to an isomorphism
\[
c_{v_{(l)}}:{\CB}_{v_{(l-1)},w, >0}^{cl} ={\CB}_{v_{(l-1)},w,>0}\isoto {\CB}_{v_{(l-1)},v_{(l)},>0}\x {\CB}_{v_{(l)},w,>0}.
\]

Recall the commutative diagram in Theorem~\ref{RichardsonCluster}. The natural quotient map $\BC[\mathring{\CB}_{v_{(l-1)}, w}] \rightarrow  \BC[\mathring{\CB}_{v_{(l)}, w}]$ given by $T_{v_{(l-1)}, v_{(l)}} = \prod_{i=1}^lT_{r_{(i)}, v_{(i)}} =1$ maps cluster variables to cluster variables, except the deleted isolated frozen cluster variable. 

Geometrically, this is given by the embedding  
\[
\CB_{v_{(l)}, w} \rightarrow \{\ast\}\times \CB_{v_{(l)}, w} \subset {\CB}_{v_{(l-1)},v_{(l)}} \x \CB_{v_{(l)}, w} \rightarrow \CB_{v_{(l-1)}, w}.
\]
Here $ \ast  \in {\CB}_{v_{(l-1)},v_{(l)}, >0} $ is the point determined by equation $T_{v_{(l-1)}, v_{(l)}} =1$; see Corollary~\ref{AtlasRegular2}.

Therefore, we have 
\begin{align*}
 \{\ast \} \x \CB_{v_{(l)}, w,>0}^{cl} &=  (\{\ast \} \x \CB_{v_{(l)}, w}) \cap c_{v_{(l)}}({\CB}_{v_{(l-1)},w,>0}^{cl}) \\
 &= (\{\ast \} \x \CB_{v_{(l)}, w}) \cap ({\CB}_{v_{(l-1)},v_{(l)},>0}\x {\CB}_{v_{(l)},w,>0}) \\
 &=  \{\ast \} \x \CB_{v_{(l)}, w, >0}.
\end{align*}

 It follows that $\CB_{v_{(l)}, w,>0}^{cl} = \CB_{v_{(l)}, w, >0}$.
\end{proof}

The discussion on total positivity extends to the twisted product. Let $\CZ_{\geq0}$ be the Hausdorff closure of $(U^-_{\geq0},\cdots, U^-_{\geq0}B^+/B^+)$ in $\CZ$ defined in \cite{BH24b}. Let $\CZ_{v,\bfw,>0}=\mathring{\CZ}_{v,\bfw}\cap \CZ_{\geq0}$. Let $\CZ_{v,\bfw, >0}^{cl}$ be the cluster positive subset defined from the cluster structure.

\begin{prop}\label{TPAgree2}
$\CZ_{v,\bfw,>0}=\CZ_{v,\bfw, >0}^{cl}$.
\end{prop}
\begin{proof}

Let $\tilde{\CB}$ be the thickening flag variety considered in \S\ref{sec:twisted:thick}. By \cite{BH24b}*{Theorem 3.4}, the isomorphism $\mathring{\Tilde{\CB}}_{v,\Tilde{w}}\isoto (\BC^\x)^{n-1}\x\mathring{\CZ}_{v,\bfw}$ restricts to an isomorphism $\Tilde{\CB}_{v,\Tilde{w},>0}\isoto \BR_{>0}^{n-1}\x\CZ_{v,\bfw,>0}$. By the same argument as the proof of Proposition \ref{TPAgree}, the identity $\Tilde{\CB}_{v,\Tilde{w},>0}=\Tilde{\CB}_{v,\Tilde{w}, >0}^{cl}$ implies $\CZ_{v,\bfw,>0}=\CZ_{v,\bfw, >0}^{cl}$.
\end{proof}

\subsection{Full rank and reddening sequences}
\label{sec:extra:fullrank}

In this section we show that the seed  constructed in Definition \ref{RichardsonSeedDefi} have full rank and admits reddening sequences. Note that Definition \ref{RichardsonSeedDefi} applies to the case of twisted product as well; see Theorem~\ref{TwistedProductCluster}.

Recall a seed $\bfs$ is said to have \textit{full rank} if the submatrix of the exchange matrix of $\bfs$ indexed by $J\x J_\uf$ has full rank \cite{BFZ05}.

\begin{prop}
\label{SeedFullRank}
\begin{enumerate}
\item The seed $\bfs(v,\bfw)$ constructed in Definition \ref{RichardsonSeedDefi} for $\mathring{\CB}_{v,w}$ has full rank. 
\item The seed $\bfs(v,\bfw)$ constructed in Definition \ref{RichardsonSeedDefi} for $\mathring{\CZ}_{v, \bfw}$ has full rank. 
\end{enumerate}
\end{prop}

\begin{proof}
We prove the first claim. The second one is similar. The base case $\mathring{\CB}_{e,w}$ is essentially \cite{BFZ05}*{Proposition 2.6}. We present a proof here for completeness.

Recall that given a reduced word $\bfw$ for $w$, the corresponding seed $\bfs$ for $\mathring{\CB}_{e,w}$ is indexed by
\[
J^+(\bfw)=\{(i,l)\,|\,i\in I,0<l\leq n_i\},
\qquad 
J^+(\bfw)_\uf=\{(i,l)\,|\,i\in I,0<l<n_i\}.
\]

Consider the submatrix with rows indexed by $\{(i,l)\,|\,i\in I,1<l \le n_i\}$ and columns indexed by $J^+(\bfw)_\uf$. 
By construction, $\ve_{(i,l+1),(i,l)}=1$, and if $(j,l') < (i,l)$, then $\ve_{(i,l+1),(j,l')}=0$. This shows that the rearranged submatrix is upper triangular with diagonal entries equal to 1. So $\bfs$ has full rank.

\cite{BFZ05}*{Lemma 3.2} shows that the full rank property is preserved under mutation. Meanwhile, in terms of the submatrix $M$ indexed by $J\x J_\uf$, freezing and deletion in the construction of $\bfs$ correspond to delete a row of $M$ and all columns of $M$ with non-zero entries in that row. By linear algebra, the full rank property is preserved. The proposition follows from induction.
\end{proof}

We refer to \cite{DK20} for the definition of maximal green sequences and reddening sequences.

\begin{prop}
\label{SeedReddening}
\begin{enumerate}
\item  The seed  $\bfs(v,\bfw)$ for  $\mathring{\CB}_{v,w}$   admits   a reddening sequence.
\item The seed  $\bfs(v,\bfw)$ for  $\mathring{\CZ}_{v, \bfw}$ admits   a reddening sequence.
\end{enumerate}
\end{prop}

\begin{proof}We shall again focus on the first case. By \cite{SW21}*{Section 4}, the seed for $\mathring{\CB}_{e,w}$ admits a maximal green sequence, which is by definition a reddening sequence. By \cite{Mul16}*{Section 3}, the property of admitting reddening sequences is preserved under mutation, freezing and deletion, so the proposition follows from induction.
\end{proof}

The full rank and reddening sequence hypotheses allows us to invoke the results in \cite{GHKK}, which answers the Fock--Goncharov conjecture in \cite{FG09}.

\begin{cor}
\begin{enumerate}
\item The upper cluster algebra $\BC[\mathring{\CB}_{v,w}]$ has a canonical basis of theta functions parametrized by the integral tropical points of the dual cluster variety.
\item The upper cluster algebra   $\BC[\mathring{\CZ}_{v, \bfw}]$ has a canonical basis of theta functions parametrized by the integral tropical points of the dual cluster variety.
\end{enumerate}
\end{cor}

\end{document}